\documentclass[10pt, a4paper, reqno]{amsart}
\usepackage{amscd}
\usepackage{dsfont}

\usepackage{diagrams}

\usepackage{amsmath}
\usepackage{amstext}
\usepackage{amsthm}
\usepackage{amssymb}
\usepackage{amsopn}
\usepackage{amssymb}
\usepackage{amsxtra}
\usepackage{amsfonts}
\usepackage{fancyhdr}
\usepackage{paralist, epsfig}
\usepackage[mathcal]{euscript}
\usepackage[english]{babel}
\usepackage[latin1]{inputenc}

\makeatletter
\g@addto@macro\normalsize{%
  \setlength\abovedisplayskip{10pt}
  \setlength\belowdisplayskip{10pt}
  \setlength\abovedisplayshortskip{5pt}
  \setlength\belowdisplayshortskip{8pt}
}

\allowdisplaybreaks

\newtheoremstyle{normal}
{5pt}
{5pt}
{\normalfont}
{}
{\bfseries}
{}
{0.4em}
{\bfseries{\thmname{#1}\thmnumber{ #2}.\thmnote{ \hspace{0.5em}(#3)\newline}}}

\newtheoremstyle{kursiv}
{5pt}
{5pt}
{\itshape}
{}
{\bfseries}
{}
{0.4em}
{\bfseries{\thmname{#1}\thmnumber{ #2}.\thmnote{ \hspace{0.5em}(#3)\newline}}}

\theoremstyle{kursiv}

\theoremstyle{normal}

\newtheorem{thm}{Theorem}[section]

\newtheorem{lem}[thm]{Lemma}
\newtheorem{prop}[thm]{Proposition}
\newtheorem{dfn}[thm]{Definition}

\newcommand{\columnvec}[2]{\genfrac{[}{]}{0pt}{}{\,#1\,}{#2}}

\renewcommand{\epsilon}{\varepsilon}

\newcommand{\id}{\operatorname{id}\nolimits}

\renewcommand{\Re}{\operatorname{Re}\nolimits}
\renewcommand{\Im}{\operatorname{Im}\nolimits}
\newcommand{\rk}{\operatorname{rk}\nolimits}

\newcommand{\ran}{\operatorname{ran}\nolimits}

\newcommand{\Cnull}{\operatorname{C_0}\nolimits}

\usepackage{setspace}
\usepackage{color}
\newcommand{\red}[1]{#1}

\definecolor{grey}{gray}{.3}

\DeclareFontEncoding{FML}{}{}
\DeclareFontSubstitution{FML}{futm}{m}{it}

\DeclareSymbolFont{gletters}{FML}{futm}{m}{it}

\DeclareMathSymbol{\PHI}{\mathord}{gletters}   {30}

\DeclareMathSymbol{\PSI}{\mathord}{gletters}  {32}
\DeclareMathSymbol{\THETA}{\mathord}{gletters}  {18}

\DeclareMathSymbol{\KAPPA}{\mathord}{gletters}  {20}

\renewcommand{\Phi}{\PHI}

\renewcommand{\Psi}{\PSI}

\renewcommand{\Theta}{\THETA}

\newarrow{Line}{-}{-}{-}{-}{-}

\begin{document}

\title{Boundary Triplets for skew-symmetric operators and\\the Generation of strongly continuous semigroups}

\author{Sven-Ake Wegner\hspace{0.5pt}\MakeLowercase{$^{\text{1}}$}}

\renewcommand{\thefootnote}{}
\hspace{-1000pt}\footnote{\hspace{5.5pt}2010 \emph{Mathematics Subject Classification}: Primary 47D06, Secondary 35G15, 47B44.}
\hspace{-1000pt}\footnote{\hspace{5.5pt}\emph{Key words and phrases}: Boundary triplet, $\Cnull$-semigroup, dissipative extensions, port-Hamiltonian system.\vspace{1.6pt}}

\hspace{-1000pt}\footnote{\hspace{0pt}$^{1}$\,University of Wuppertal, School of Mathematics and Natural Sciences, Gau\ss{}stra\ss{}e 20, 42119 Wuppertal, Germany,\linebreak\phantom{x}\hspace{1.2pt}Phone:\hspace{1.2pt}\hspace{1.2pt}+49\hspace{1.2pt}(0)202\hspace{1.2pt}/\hspace{1.2pt}439\hspace{1.2pt}-\hspace{1.2pt}2590, Fax:\hspace{1.2pt}\hspace{1.2pt}+49\hspace{1.2pt}(0)\hspace{1.2pt}202\hspace{1.2pt}/\hspace{1.2pt}439\hspace{1.2pt}-\hspace{1.2pt}3724, E-Mail: wegner@math.uni-wuppertal.de.\vspace{1.6pt}}

\begin{abstract} We give a self-contained and streamlined exposition of a generation theorem for $\Cnull$-semigroups based on the method of boundary triplets. We apply this theorem to port-Hamiltonian systems where we discuss recent results appearing in stability and control theory. We give detailed proofs and require only a basic knowledge of operator and semigroup theory.
\end{abstract}

\maketitle


\section{Introduction}\label{SEC-1}

\smallskip

It is a prominent result from the theory of unbounded operators that there is a one-to-one correspondence between unbounded self-adjoint and bounded unitary operators. This correspondence is a key ingredient for the characterization of those symmetric operators which allow for a self-adjoint extension, established by von Neumann in 1929.

\smallskip

Given a skew-symmetric operator $A\colon D(A)\subseteq X\rightarrow X$ on a Hilbert space $X$, a variant of the correspondence above gives rise to a bijection between dissipative extensions and a certain class of contractive operators. This in turn yields a parametrization of those extensions generating a $\Cnull$-semigroup of contractions
$$
\bigg\{\hspace{-3pt}\begin{array}{c}
B\colon D(B)\rightarrow X \text{ such that }A\subseteq B \text{ \& } B\text{ ge-}\\
 \text{nerates a $\Cnull$-semigroup of contractions}
\end{array}\hspace{-3pt}\bigg\}
\;\stackrel{1\text{-}1}{\longleftrightarrow}\;
\bigg\{\hspace{-3pt}\begin{array}{c}
K\colon X\rightarrow X \text{ is contractive \& }\\
 K(1-A)=(1+A) \text{ on } D(A)
\end{array}\hspace{-3pt}\bigg\}
$$
which was established by Phillips in 1959. Above, the contractions are defined on the initial Hilbert space $X$. The fact that on $\ran(1-A)$ the operator $K$ is completely determined by $A$ suggests that  a smaller domain can be enough for the parametrization. In addition, classical applications to boundary value problems---where the extension problem amounts to determine the correct boundary conditions yielding a generator---also suggest that the degrees of freedom in the extension problem can be captured by operators defined on a much smaller, possibly even finite-dimensional, space.

\smallskip

Boundary triplets are the technical incarnation of this idea. Beginning in the late 1960's their theory was developed for the case of a symmetric operator $A$ by making use of linear relations in the statements as well as in the proofs. We refer to Gorbachuk, Gorbachuk \cite[p.\,320]{GG} for detailed historical information. Nowadays, half a century later, boundary triplets appear in several very active research areas. As a sample of recent papers in this direction we mention Derkach, Malamud \cite{DM1991, DM1995}, Malamud \cite{M1992},  Behrndt, Langer \cite{BL2007, BL2012},  Behrndt, Hassi, de Snoo, Wietsma \cite{BHdSW2011}, Arlinski\u\i{} \cite{A2012}, Derkach, Hassi, Malamud, de Snoo \cite{DHMdS2006, DHMdS2009, DHMdS2012} and direct the reader also to the references therein. \red{We point out that the trivial boundary triplet, i.e., the space in the parametrization is zero, corresponds to the situation where the initial operator $A$ is self-adjoint. This exhibits a link between the above and Stone's 1930 theorem on the one-to-one correspondence between self-adjoint operators and unitary $\Cnull$-semigroups.} 

\smallskip

Our aim in this article is to give a self-contained and streamlined exposition of the technique of boundary triplets in the case of a skew-symmetric operator $A$. Our account is mainly based on the material presented in the books of Gorbachuk, Gorbachuk \cite[Chapter 3.1]{GG} and Schm\"udgen \cite[Section 14]{S}. The principal difference of our approach is that we develop from the beginning all technical details in the skew-symmetric situation and that we restrict ourselves to dissipative operators instead of dissipative relations. As it turns out, it is then possible to avoid relations also in all proofs. We believe that this approach is of value in particular for beginners in the subject who have a semigroup background or are heading in a semigroup direction.

\smallskip

In addition to a rigorous exposition of boundary triplets in the skew-symmetric situation and without the use of relations, we give a modern application of the method of boundary triplets. Also here our exposition is very detailed and intended for beginners in this area. The content is taken from Le Gorrec, Zwart, Maschke \cite{GZM2005} who applied the method of boundary triplets to port-Hamiltonian systems. Other applications in systems theory are for instance given in Malinen, Staffans \cite{MS2007}, Villegas, Zwart, Le Gorrec, Maschke \cite{VZGM2009}, Arov, Kurula, Staffans \cite{AKS2012} or Kurula, Zwart \cite{KZ2015}. A general reference for port-Hamiltonian systems is the book by Jacob, Zwart \cite{JZ} from whence we also took several arguments used in Section \ref{SEC-6}. Recent results, further improving the current generation theorem, are contained in Augner, Jacob \cite{AJ2014} and Jacob, Morris, Zwart \cite{JMZ2015}.

\smallskip

We require the reader to be familiar with the basics of functional analysis as they are usually contained in a first course on this topic, i.e., Banach spaces, linear operators, Hahn-Banach theorem, closed graph and open mapping theorems, Hilbert spaces, orthogonality, Riesz-Fr\'{e}chet representation theorem etc. For references concerning those basics we use the textbook of Conway \cite{C}. In addition we require some facts about unbounded operators but not more than the definition of the adjoint and some of its consequences. We advise the reader to consult \cite[Section X.1, p.\,303--308]{C} if necessary. Finally, we require some knowledge on operator semigroups, but again not more than the definition, the notion of a generator and the Hille-Yosida theorem. Here, we recommend to use the introduction given by Rudin \cite[p.\,375--385]{R}.

\smallskip

This article is organized as follows. In Section \ref{SEC-2} we first collect basic facts about dissipative operators and the Cayley transform. Then we prove that all dissipative extensions of a given skew-symmetric operator are modulo a sign restrictions of its adjoint. In Section \ref{SEC-3} we give a version of the Lumer-Phillips theorem which fits to our situation and whose proof relies only on the prerequisites we mentioned above. Readers familiar with the Lumer-Phillips theorem can skip this section, but are advised to browse the corresponding comments in Section \ref{SEC-7}. Section \ref{SEC-4} contains the parametrization of generators via boundary triplets. The characterization of those operators for which a boundary triplet exists is given in Section \ref{SEC-5}. Readers who are more focused on applications may also skip this section and move on directly to Section \ref{SEC-6} on port-Hamiltonian systems.

\smallskip

The dependencies between the sections are visualized below.
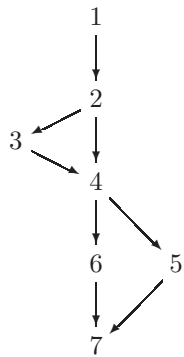
\begin{diagram}[height=1.55em,width=1.5em]
   & & 1 & & \\ 
   & & \dTo  & & \\ 
   & & 2 & & \\ 
3  &\ldTo(2,1) & \dTo & & \\ 
   &\rdTo(2,1) & 4  & & \\ 
   & & \dTo &\rdTo & \\ 
   & & 6 & & 5\\ 
   & & \dTo  &\ldTo & \\ 
   & & 7 & & \\ 
\end{diagram}

\smallskip

In Sections \ref{SEC-2}--\ref{SEC-6} we omit on purpose all references to the original articles. We believe that this makes the text more accessible and easier to read---in particular for beginners. Later, in Section \ref{SEC-7}, we go through the article again and collect all references omitted before. Moreover, we add comments and remarks on related work.

\bigskip


\section{Dissipative Operators}\label{SEC-2}

In what follows we discuss the basic theory of skew-symmetric operators, dissipative extensions and the Cayley transform. We fix a complex Hilbert space $X$ and we assume that its inner product $\langle{}\cdot,\cdot\rangle{}$ is linear in the first argument and conjugate linear in the second argument. By a linear operator in $X$ we mean an operator $A\colon D(A)\rightarrow X$ with a linear subspace $D(A)\subseteq X$ as domain. The symbol \textquotedblleft{}$A$\textquotedblright{} for an operator comprises that a domain is given and fixed even if not mentioned explicitly. If $A$ is an operator in $X$ and $\lambda$ a complex number, then we write for simplicity $\lambda-A$ instead of $\lambda\cdot\id_X-A$ and we put $D(\lambda-A)=D(A)$. \red{The symbol \textquotedblleft{}$\oplus$\textquotedblright{} denotes a direct sum, which needs not necessarily to be orthogonal. If it is orthogonal, we will state this explicitly.}

\smallskip

\begin{dfn}\label{dissipative} Let $A$ be a densely defined operator.
\begin{compactitem}\vspace{3pt}
\item[(i)] The operator $A$ is \emph{dissipative}, if $\Re\langle{}Ax,x\rangle{}\leqslant0$ holds for all $x\in D(A)$.\vspace{3pt}
\item[(ii)] The operator $A$ is \emph{maximal dissipative}, if $A$ is dissipative and no proper dissipative extensions of $A$ exists, i.e., \red{if for every dissipative operator $B$,  $A\subseteq B$ implies $A=B$.}
\end{compactitem}
\end{dfn}

\smallskip

Below we summarize basic properties of dissipative and maximal dissipative operators. \red{In Lemma \ref{LEM}(vi) we consider adjoints of operators.} We advise the reader to recall, if necessary, the definition of $A^{\star}\colon D(A^{\star})\rightarrow X$ for a given densely defined operator $A\colon D(A)\rightarrow X$ and to review its basic properties, in particular the fundamental relation $\langle{}Ax,y\rangle{}=\langle{}x,A^{\star} y\rangle{}$ valid for $x\in D(A)$ and $y\in D(A^{\star})$, see \cite[Section X, \S 1]{C}.

\smallskip

\begin{lem}\label{LEM} Let $A$ be a densely defined operator.
\begin{compactitem}\vspace{3pt}
\item[(i)] If $A$ is dissipative, then $A$ is closable and its closure is again dissipative.\vspace{3pt}
\item[(ii)] If $A$ is maximal dissipative, then $A$ is closed.\vspace{3pt}
\item[(iii)] If $A$ is dissipative, then there exists a maximal dissipative extension $B$ of $A$.\vspace{3pt}
\item[(iv)] $A$ is dissipative if and only if $\lambda\|x\|\leqslant\|(\lambda-A)x\|$ holds for every $x\in D(A)$ and every $\lambda>0$.\vspace{3pt}
\item[(v)] If $A$ is dissipative and closed, then $\ran(\lambda-A)\subseteq X$ is closed for every \red{$\lambda>0$}.\vspace{3pt}
\red{\item[(vi)] If $A$ is dissipative and closed, then $X=\ran(\lambda-A)\oplus\ker(\lambda-A^{\star})$ holds with an orthogonal sum for every $\lambda>0$.}
\end{compactitem}
\end{lem}
\begin{proof} (i) Let $(x_n)_{n\in\mathbb{N}}\subseteq D(A)$ and $y\in X$ with $x_n\rightarrow0$ and $Ax_n\rightarrow y$ be given. For $x\in D(A)$ and $\lambda\in\mathbb{C}$ we have $\Re\langle{}A(x+\lambda{}x_n),x+\lambda{}x_n\rangle{}\leqslant0$ since $A$ is dissipative. Consequently, also for the limit $\Re\langle{}Ax+\lambda{}y,x\rangle{}=\Re\langle{}Ax,x\rangle{}+\Re\lambda{}\langle{}y,x\rangle{}\leqslant0$ holds. We get that $\Re\lambda{}\langle{}y,x\rangle{}\leqslant|\Re\langle{}Ax,x\rangle{}|$ holds for every $\lambda\in\mathbb{C}$ and every $x\in D(A)$. For fixed $y$ and $x$ the last estimate can only be true for all $\lambda$ if $\langle{}y,x\rangle{}=0$. But since $D(A)\subseteq X$ is dense the latter in turn can only be true if $y=0$. This shows that $A$ is closable, see \cite[comments after X.1.3]{C}. The closedness, together with the continuity of $\Re\langle\cdot,\cdot\rangle{}$, implies that the closure of $A$ is dissipative.

\smallskip

(ii) If $A$ is not closed then its closure is a proper dissipative extension in view of (i).

\smallskip

(iii) The existence of a maximal dissipative extension follows by using Zorn's lemma.

\smallskip

(iv) \textquotedblleft{}$\Rightarrow$\textquotedblright{} For $\lambda>0$ and $x\in D(A)$ we use the dissipativity to get
$$
2\lambda\|x\|^2 \leqslant  2\lambda\langle{}x,x\rangle{}-(\langle{}Ax,x\rangle{}+\langle{}x,Ax\rangle{}) 
= \langle{}(\lambda-A)x,x\rangle{}+\langle{} x,(\lambda-A)x\rangle{}\leqslant2\|x\|\|(\lambda-A)x\|
$$
and obtain $\lambda\|x\|\leqslant\|(\lambda-A)x\|$ as desired.

\smallskip

\textquotedblleft{}$\Leftarrow$\textquotedblright{} Let $x\in D(A)$ and $\lambda>0$ be given. By $\lambda\|x\|\leqslant\|(\lambda-A)x\|$ we have
$$
\|\lambda{}x\|^2\leqslant\|(\lambda-A)x\|^2 = \|\lambda{}x\|^2+\|Ax\|^2-2\Re\langle{}Ax,\lambda{}x\rangle{}
$$
which implies $2\lambda{}\Re\langle{}Ax,x\rangle{}\leqslant\|Ax\|^2$ and thus $\Re\langle{}Ax,x\rangle{}\leqslant\frac{1}{2\lambda}\|Ax\|^2$. Since $\lambda>0$ was arbitrary, we \red{proved that} $\Re\langle{}Ax,x\rangle{}\leqslant0$.

\smallskip

(v) Fix $\lambda>0$ and let $(x_n)_{n\in\mathbb{N}}\subseteq D(A)$ with $(\lambda-A)x_n\rightarrow y\in X$ be given. We claim that $y$ is in the range of $\lambda-A$. By (iv) we get $\lambda\|x_n-x_m\|\leqslant\|(\lambda-A)x_n-(\lambda-A)x_m\|$ for arbitrary $n,\,m\in\mathbb{N}$ and conclude that $(x_n)_{n\in\mathbb{N}}$ is a Cauchy sequence and thus converges to some $x\in X$. We have $x_n\rightarrow x$ and $(\lambda-A)x_n\rightarrow y$. Since $A$ and therefore also $\lambda-A$ is closed it \red{follows that} $(\lambda-A)x=y$ and we \red{showed that} $y\in\ran(\lambda-A)$.

\smallskip

(vi) Straightforward computations based on the definition of the adjoint show that $(\lambda-A)^{\star}=\lambda-A^{\star}$ holds. By \cite[Proposition X.1.13]{C} we thus have $\ker(\lambda-A^{\star})=\ran(\lambda-A)^{\perp}$. Since $\ran(\lambda-A)\subseteq X$ is closed in view of (v) we obtain the desired decomposition by standard Hilbert space theory, see \cite[Exercise 3 on p.~11]{C}.
\end{proof}

\smallskip

In what follows operators of the type defined below will always be our starting point.

\smallskip

\begin{dfn}\label{skew-symmetric} A densely defined operator $A$ is said to be \textit{skew-symmetric} if $A\subseteq-A^{\star}$ holds.
\end{dfn}

\smallskip

\red{If $A$ is skew-symmetric, then $A$ and $-A$ are dissipative.} In particular, the properties established for dissipative operators in Lemma \ref{LEM} are valid for skew-symmetric ones. From now on we stick to the notation $D(-A^{\star})$ for the domain of $-A^{\star}$, which is by definition equal to $D(A^{\star})$.

\smallskip

\begin{lem}\label{SIMPLE-DECOMP} Let $A$ be densely defined, closed and \red{skew-symmetric}. Let $\lambda>0$ and $\epsilon=\pm1$ be fixed. Then we have
$$
X=\ran(\lambda-\epsilon{}A)\oplus\ker(\lambda-\epsilon{}A^{\star})
$$
\red{where the sum is orthogonal}.
\end{lem}
\begin{proof} \red{If $\epsilon=+1$ then the statement is a specialization of Lemma \ref{LEM}(vi). If $\epsilon=-1$, we apply Lemma \ref{LEM}(vi) to $-A$, which is dissipative as $A$ is skew-symmetric, and observe that $(-A)^{\star}=-A^{\star}$ establishes the claim.}
\end{proof}

\smallskip

\begin{lem}\label{DECOMP-LEM} Let $A$ be a densely defined, closed and skew-symmetric operator. Then
$$
D(-A^{\star})=D(A)\oplus\ker(1-A^{\star})\oplus\ker(1+A^{\star})
$$
and $-A^{\star}x=Ax_1-x_2+x_3$ holds for $x=x_1+x_2+x_3\in D(-A^{\star})$ with $x_1\in D(A)$, $x_2\in\ker(1-A^{\star})$ and $x_3\in\ker(1+A^{\star})$. \red{Above, the direct sum is not necessarily orthogonal.}
\end{lem}
\begin{proof} All three summands are subspaces of $D(-A^{\star})$. We show that their sum is direct. Let $x_1\in D(A)$, $x_2\in\ker(1-A^{\star})$ and $x_3\in\ker(1+A^{\star})$ satisfy $x_1+x_2+x_3=0$. We compute $0=(1-A^{\star})(x_1+x_2+x_3)=(1+A)x_1\red{+0+(1+1)x_3}$. Here, $(1+A)x_1\in\ran(1+A)$ and $2x_{\red{3}}\in\ker(1+A^{\star})$. The latter two spaces are orthogonal. Therefore, $x_{\red{3}}=0$ and $(1+A)x_1=0$, i.e.,
$$
0=\langle{}(1+A)x_1,x_1\rangle{}=\langle{}x_1,x_1\rangle{}+\langle{}Ax_1,x_1\rangle{}.
$$
Since $A$ is skew-symmetric we have $\Re\langle{}Ax_1,x_1\rangle{}=0$. Thus, taking real parts on both sides of the last equation gives $\langle{}x_1,x_1\rangle{}=0$ and thus $x_1=0$. It \red{follows that} $x_{\red{2}}=0$ and we are done.

\smallskip

In order to show $D(-A^{\star})\subseteq{}D(A)+\ker(1-A^{\star})+\ker(1+A^{\star})$ we take $x\in D(-A^{\star})$ and decompose $(1+A^{\star})x\in X$ according to Lemma \ref{SIMPLE-DECOMP} with $\lambda=1$ and $\epsilon=1$, i.e., we take $x_0\in D(A)$ and $y_0\in\ker(1-A^{\star})$ such that $(1+A^{\star})x=(1-A)x_0+y_0$. Now we compute
\begin{eqnarray*}
(1+A^{\star})(x-x_0-{\textstyle\frac{y_0}{2}}) & = & (1-A)x_0 + y_0 - (1+A^{\star})x_0-(1+A^{\star}){\textstyle\frac{y_0}{2}}\\
& = & (1-A)x_0 +y_0 - (1-A^{\star})x_0 - (1+1){\textstyle\frac{y_0}{2}}\\
& = & 0,
\end{eqnarray*}
which shows $x-x_0-{\textstyle\frac{y_0}{2}}\in\ker(1+A^{\star})$. We put $x_1=-x_0$, $x_2=-{\textstyle\frac{y_0}{2}}$ and $x_3=x+x_0+{\textstyle\frac{y_0}{2}}$ and thus $x=x_1+x_2+x_3$ holds with $x_1\in D(A)$, $x_2\in\ker(1-A^{\star})$ and $x_3\in\ker(1+A^{\star})$.
\end{proof}

\smallskip

In \cite[Section X, \S 3]{C} and also in \cite[Chapter 13]{R} the reader can find a detailed study of the Cayley transform in the context of symmetric operators. Below we adjust its definition for our case of skew-symmetric and dissipative operators.

\smallskip

\begin{dfn}\label{DFN-CT} Let $A$ be an operator. The \textit{Cayley transform} $C_A$ of $A$ is defined via
$$
C_A=(1+A)(1-A)^{-1}\colon\ran(1-A)\rightarrow X
$$
provided that $1-A\colon D(A)\rightarrow X$ is injective.
\end{dfn}

\smallskip

Below we establish properties of $C_A$.

\smallskip

\begin{lem}\label{LEM-3} Let $A$ be densely defined and dissipative. Then the Cayley transform $C_A$ is well-defined with $\ran C_A=\ran(1+A)$. \red{The map $1+C_A\colon\ran(1-A)\rightarrow D(A)$ is bijective and we have $C_A=2(1-A)^{-1}-1$ on $\ran(1-A)$} and $A=-(1-C_A)(1+C_A)^{-1}$ on $D(A)$. The Cayley transform $C_A$ is contractive. If $A$ is skew-symmetric, then it is isometric.
\end{lem}
\begin{proof}\red{The operator $1-A\colon D(A)\rightarrow\ran(1-A)$ is bijective by Lemma \ref{LEM}(iv). Therefore $C_A$ is well-defined and $\ran C_A=\ran(1+A)$ holds.}

\smallskip

\red{For $y\in\ran(1-A)$ we select $x\in D(A)$ with $(1-A)x=y$ and compute
$$
(2(1-A)^{-1}-1)y = 2(1-A)^{-1}y-y = 2x-(1-A)x = (1+A)x = (1+A)(1-A)^{-1}y = C_Ay
$$
which shows that $C_A=2(1-A)^{-1}-1$ holds. It follows that $(1+C_A)/2=(1-A)^{-1}$ and we see that $1+C_A\colon\ran(1-A)\rightarrow D(A)$ is bijective. Now we obtain $2(1+C_A)^{-1}x=(1-A)x$ for each $x\in D(A)$ and thus
$$
Ax = x-2(1+C_A)^{-1}x = (1+C_A)(1+C_A)^{-1}x-2(1+C_A)^{-1}x = -(1-C_A)(1+C_A)^{-1}x
$$
for $x\in D(A)$.} Finally, we estimate
\begin{eqnarray*}
\|y+Ay\|^2 & = & \|y\|^2 + \|Ay\|^2+\langle{}Ay,y\rangle{}+\langle{}y,Ay\rangle{}\\
           & \leqslant & \|y\|^2+\|Ay\|^2-\langle{}Ay,y\rangle{}-\langle{}y,Ay\rangle{}\\
           & = & \|y-Ay\|^2.
\end{eqnarray*}
Since $C_Ax=y+Ay$ and $y-Ay=x$ are valid the above implies $\|C_Ax\|\leqslant\|x\|$. If $A$ is skew-symmetric then $\langle{}Ay,y\rangle{}+\langle{}y,Ay\rangle{}=0$ holds for $y\in D(A)$ and thus we obtain $\|C_Ax\|=\|x\|$.
\end{proof}

\smallskip

\begin{prop}\label{PROP-1} Let $A$ be densely defined, closed and skew-symmetric, i.e., $A\subseteq-A^{\star}$. Let $B$ be a dissipative extension of $A$. Then $B\subseteq-A^{\star}$ holds.
\end{prop}
\begin{proof} By Lemma \ref{LEM-3} the Cayley transforms
\begin{equation*}
C_A=(1+A)(1-A)^{-1}\colon\ran(1-A)\rightarrow X \; \text{ and } \; C_B=(1+B)(1-B)^{-1}\colon\ran(1-B)\rightarrow X
\end{equation*}
are both well-defined, $C_B$ is contractive and $C_A$ is isometric. By construction we have $C_A\subseteq C_B$ and from Lemma \ref{SIMPLE-DECOMP} with $\lambda=1$ we obtain
$$
X=\ran(1-A)\oplus\ker(1-A^{\star}) \; \text{ and } \; X=\ran(1+A)\oplus\ker(1+A^{\star})
$$
\red{with orthogonal sums. We claim that there exists $C\colon \ker(1-A^{\star})\cap\ran(1-B)\rightarrow\ker(1+A^{\star})$ such that
$$
C_Bx=C_Ax_1+Cx_2
$$
holds for every $x\in\ran(1-B)$ where $x=x_1+x_2$ and $x_1\in\ran(1-A)$, $x_2\in\ker(1-A^{\star})$. We observe that $C_Bx=C_Bx_1+C_Bx_2=C_Ax_1+C_Bx_2$ holds for $x_1$ and $x_2$ as above.} Therefore it is enough to show that $C_B(\ran(1-B)\cap\ker(1-A^{\star}))\subseteq\ker(1+A^{\star})$ holds. Let $z\in\ran(1-B)\cap\ker(1-A^{\star})$, i.e., $z\in\ran(1-B)$ with $z\perp{}\ran(1-A)$. Let $y\in\ran(1-A)$ and $\lambda\in\mathbb{C}$. We have
$$
\|\lambda{}C_By\|^2+\|C_Bz\|^2+2\Re\langle{}\lambda{}C_By,C_Bz\rangle{}=\|C_B(\lambda{}y+z)\|^2\leqslant \|\lambda{}y+z\|^2 = \|\lambda{}y\|^2+\|z\|^2
$$
where the estimate holds since $C_B$ is contractive and the last equality holds since $z\perp{}y$ is valid by our assumptions. \red{From Lemma \ref{LEM-3} we know that $\|C_By\|=\|C_Ay\|=\|y\|$ holds. Hence, the above inequality leads to
$$
2\Re\lambda{}\langle{}C_By,C_Bz\rangle{} \leqslant \|\lambda{}y\|^2+\|z\|^2-\|\lambda{}C_By\|^2-\|C_Bz\|^2 = \|z\|^2-\|C_Bz\|^2
$$
which can only be true for every $\lambda\in\mathbb{C}$, if $\langle{}C_By,C_Bz\rangle{}=0$ holds. Hence $C_Bz\perp{}C_By$ holds for every $y\in\ran(1-A)$. As $C_B(\ran(1-A))=C_A(\ran(1-A))=\ran C_A = \ran(1+A)$ holds by Lemma \ref{LEM-3} we get $C_Bz\perp{}\ran(1+A)$ which implies $C_Bz\in\ker(1+A^{\star})$ in view of Lemma \ref{SIMPLE-DECOMP}. Now we put
$$
C:=C_B|_{\ker(1-A^{\star})\cap\ran(1-B)}
$$ 
which establishes the claim.}

\smallskip

Now we show that $B\subseteq{}-A^{\star}$ holds. Let $x\in D(B)$ and put $y=(1-B)x$. According to the latter two paragraphs we can decompose $y=y_1+y_2$ with $y_1\in\ran(1-A)$ and $y_2\in\ker(1-A^{\star})$ such that $C_By=C_Ay_1+Cy_2$ where $C_Ay_1\in\ran(1+A)$ and $Cy_2\in\ker(1+A^{\star})$ holds. \red{Using the formula $C_B=2(1-B)^{-1}-1$, see Lemma \ref{LEM-3}, we get
$$
x = (1-B)^{-1}y = {\textstyle\frac{1}{2}}(1+C_B)y = {\textstyle\frac{1}{2}}(y_1+y_2+C_Ay_1+Cy_2)={\textstyle\frac{1}{2}}(1+C_A)y_1+{\textstyle\frac{y_2}{2}}+{\textstyle\frac{Cy_2}{2}}
$$
where, also by Lemma \ref{LEM-3}, $(1+C_A)y_1\in D(A)$ holds.} Consequently,
$$
x\in D(A)\oplus{}\ker(1-A^{\star})\oplus{}\ker(1+A^{\star})=D(-A^{\star})
$$
where the equality follows from Lemma \ref{DECOMP-LEM}. This establishes $D(B)\subseteq D(-A^{\star})$. Applying Lemma \ref{DECOMP-LEM} \red{and Lemma \ref{LEM-3}} again, we obtain for $x\in D(B)$ as above
$$
-A^{\star}x={\textstyle\frac{1}{2}}A(1+C_A)y_1-{\textstyle\frac{y_2}{2}}+{\textstyle\frac{Cy_2}{2}}={\textstyle\frac{y_1}{2}}+{\textstyle\frac{C_Ay_1}{2}}-{\textstyle\frac{y_2}{2}}+{\textstyle\frac{Cy_2}{2}}.
$$
From $C_By=(1+B)x=x+Bx$ and $x={\textstyle\frac{1}{2}}(y+C_By)$ we get on the other hand
$$
Bx = C_By-x = C_By-{\textstyle\frac{1}{2}}(y+C_By) = {\textstyle\frac{1}{2}}(C_By-y) = {\textstyle\frac{C_Ay_1}{2}}+{\textstyle\frac{Cy_2}{2}}-{\textstyle\frac{y_1}{2}}-{\textstyle\frac{y_2}{2}}
$$
which shows $Bx=-A^{\star}x$ and finishes the proof.
\end{proof}

\bigskip


\section{The Lumer-Phillips Theorem}\label{SEC-3}

\smallskip

In this section we give a proof of a version of the Lumer-Phillips theorem which fits precisely to the situation where we apply it in Section \ref{SEC-4} and to our terminology of maximal dissipativity. Those readers familiar with one or another version of the Lumer-Phillips theorem are advised to consult the corresponding paragraph in Section \ref{SEC-7}.

\smallskip

Here, our aim is to give a proof which only relies on the theory of semigroups developed in \cite[pp.~375]{R}.

\begin{thm}\label{Lumer-Phillips} Let $X$ be a Hilbert space and let $A$ be densely defined. Then $A$ generates a $\Cnull$-semigroup of contractions on $X$ if and only if $A$ is maximal dissipative.
\end{thm}
\begin{proof} \textquotedblleft{}$\Rightarrow$\textquotedblright{} Let $A$ generate a $\Cnull$-semigroup of contractions. Then by \cite[13.35(f)]{R} the map $\lambda-A$ has an inverse in $L(X)$  for every $\lambda>0$ and the estimate $\|(\lambda-A)^{-1}y\|\leqslant{}{\textstyle\frac{1}{\lambda}}\|y\|$ holds for every $\lambda>0$ and every $y\in X$. By Lemma \ref{LEM}(iv) it follows that $A$ is dissipative. In addition it follows that $\lambda-A$ is surjective. Since by Lemma \ref{LEM}(iv) every dissipative operator is in particular injective, $A$ cannot have a proper dissipative extension and thus is maximal dissipative.

\smallskip

\textquotedblleft{}$\Leftarrow$\textquotedblright{} Let $A$ be maximal dissipative. We use the Hille-Yosida theorem \cite[13.37]{R}. Let $\lambda>0$ be given. By Lemma \ref{LEM}(iv), $\lambda-A\colon D(A)\rightarrow X$ is injective. By Lemma \ref{LEM}(ii) and Lemma \ref{LEM}(v) it follows that $\ran(\lambda-A)\subseteq X$ is closed. Assume that the latter is a proper inclusion. Then we can write $X=\ran(\lambda-A)\oplus\ker(\lambda-A^{\star})$ with an orthogonal direct sum, cf.~Lemma \ref{SIMPLE-DECOMP}. We define
$$
B\colon D(B)\rightarrow X \; \text{ with } \; D(B)=D(A)+\ker(\lambda-A^{\star}) \; \text{ and } \; B(x_1+x_2)=Ax_1-\lambda{}x_2.
$$
The operator $B$ is well-defined: If $x_1+x_2=0$ then $x_2=-x_1\in D(A)$ and $Ax_1=-Ax_2$ together with $A^{\star}x_2=\lambda{}x_2$ implies $\langle{}Ax_1,x_1\rangle{} = \langle{}Ax_2,x_2\rangle{} = \langle{}x_2,A^{\star}x_2\rangle{}=\langle{}x_2,\lambda{}x_2\rangle{} = \lambda{}\langle{}x_2,x_2\rangle{}\geqslant0$. On the other hand $\lambda{}\langle{}x_2,x_2\rangle{}=\Re\langle{}Ax_2,x_2\rangle{}\leqslant0$. Therefore, $x_2=0$ and then $x_1=0$ follows. For $x_1\in D(A)$ and $x_2\in\ker(\lambda-A^{\star})$ we compute $\langle{}Ax_1,x_2\rangle{} = \langle{}x_1,A^{\star}x_2\rangle{} = \langle{}x_1,\lambda{}x_2\rangle{}=\lambda{}\langle{}x_1,x_2\rangle{}$ and obtain
\begin{eqnarray*}
\langle{}B(x_1+x_2),x_1+x_2\rangle{} & = & \langle{}Ax_1,x_1\rangle{} + \langle{}Ax_1,x_2\rangle{} + \langle{}-\lambda{}x_2,x_1\rangle{} + \langle{}-\lambda{}x_2,x_2\rangle{} \\
& = & \langle{}Ax_1,x_1\rangle{} + \lambda\langle{}x_1,x_2\rangle{} - \lambda{}\langle{}x_2,x_1\rangle{} - \lambda{}\langle{}x_2,x_2\rangle{}\\
& = & \langle{}Ax_1,x_1\rangle{} + i\lambda\Im\langle{}x_1,x_2\rangle{}   -\lambda{}\|x_2\|^2
\end{eqnarray*}
which implies $\Re\langle{}B(x_1+x_2),x_1+x_2\rangle{}=\Re\langle{}Ax_1,x_1\rangle{}-\lambda\|x_2\|^2\leqslant0$. Thus, $B$ is dissipative. By construction $B$ is a proper extension of $A$ which contradicts the maximal dissipativity of $A$ and thus establishes that $\lambda-A\colon D(A)\rightarrow X$ is also surjective. But then the estimate which we obtained above from Lemma \ref{LEM}(iv) provides $\|(\lambda-A)^{-1}y\|\leqslant \frac{1}{\lambda}\|y\|$ for every $y\in X$. Finally, for every $n\in\mathbb{N}$ we have $\|(\lambda-A)^{-m}\|\leqslant\|(\lambda-A)^{-1}\|^m \leqslant \lambda^{-m}$ which establishes the condition in the Hille-Yosida Theorem \cite[13.37]{R} with $C=1$ and $\gamma=0$. The proof of the latter theorem, cf.~\cite[p.~382, first paragraph]{R} shows that the in this case the operators in the semigroup are all contractions.
\end{proof}

\bigskip


\section{Boundary Triplets and Generation}\label{SEC-4}

\smallskip

This section is the heart of the current article. Below we give the parametrization of generators via contractions defined on an auxiliary Hilbert space employing a boundary triplet. We mainly follow the \red{approach} of Gorbachuk, Gorbachuk \cite{GG} and Schm\"udgen \cite{S} but start with a skew-symmetric operator and completely avoid the use of relations.

\smallskip

\begin{dfn}\label{dfn-BT} Let $X$ be a Hilbert space and $A$ be a densely defined and skew-symmetric operator in $X$. A \textit{boundary triplet} for $A$ is a triple $(H,\Gamma_1,\Gamma_2)$ consisting of a Hilbert space $H$ and linear and maps $\Gamma_1$, $\Gamma_2\colon D(-A^{\star})\rightarrow H$ such that the two conditions\vspace{2.5pt}
{\flushleft$\begin{array}{r@{}l}\vspace{3.5pt}
\text{(BT-1)}\;\;\;\;&{}\forall\;x,\,y\in D(-A^\star)\colon \langle{}A^{\star}x,y\rangle_X+\langle{}x,A^{\star}y\rangle_X=\langle{}\Gamma_1x,\Gamma_2y\rangle_H+\langle{}\Gamma_2x,\Gamma_1y\rangle_H\\\vspace{3pt}
\text{(BT-2)}\;\;\;\;&{}\forall\:y_1,\,y_2\in H \;\exists\:x\in D(-A^{\star})\colon \Gamma_1x=y_1 \text{ and } \Gamma_2x=y_2
\end{array}$}\vspace{2.5pt}

are satisfied.
\end{dfn}

\smallskip

We note that in the literature the above triple---or its symmetric counterpart---is usually referred to as a boundary triplet \textquotedblleft{}for $A^{\star}$\textquotedblright{}. The condition (BT-1) is often said to be an \textquotedblleft{}abstract Green\red{'s} identity\textquotedblright{}.

\smallskip

Now we state the main result of this section.

\smallskip

\begin{thm}\label{MAIN} Let $A$ be a densely defined, closed, skew-symmetric operator in a Hilbert space $X$ and let $(H,\Gamma_1,\Gamma_2)$ be a boundary triplet for $A$. Let $B$ be an extension of $A$. Then $B$ generates a $\Cnull$-semigroup of contractions on $X$ if and only if there exists a uniquely determined $K\in L(H)$ with $\|K\|_{L(H)}\leqslant1$ such that $D(B)=\{x\in D(-A^{\star})\:;\:(K-1)\Gamma_1x-(K+1)\Gamma_2x=0\}$ and $B=-A^{\star}|_{D(B)}$.
\end{thm}

\smallskip

Note that we \textit{explicitly} construct a bijection and its inverse \red{below}. For applications the technical version, Proposition \ref{PROP-3-NEU}, of the theorem above can be more handy.

\smallskip

For the rest of this section we fix a densely defined, closed and skew-symmetric operator $A$ and a boundary triplet $(H,\Gamma_1,\Gamma_2)$ for $A$. We start with the following fundamental relations between operators and boundary triplets.

\smallskip

\begin{lem}\label{LEM-0} Let $A$ be as above.
\begin{compactitem}\vspace{2pt}
\item[(i)] $D(A)=\{x\in D(-A^{\star})\:;\:\Gamma_1x=\Gamma_2x=0\}$.\vspace{3pt}
\item[(ii)] $\forall\:x\in D(-A^{\star})\colon\Re\langle{}-A^{\star}x,x\rangle{}_X=-\Re\langle{}\Gamma_1x,\Gamma_2x\rangle{}_H$.
\end{compactitem} 
\end{lem}
\begin{proof} (i) \textquotedblleft{}$\subseteq$\textquotedblright{} Let $x\in D(A)$. We consider the points $\Gamma_2x$ and $\Gamma_1x$ in $H$ and select $y\in D(\red{-}A^{\star})$ with $\Gamma_1y=\Gamma_2x$ and $\Gamma_2y=\Gamma_1x$. According to (BT-1) we have
$$
\langle{}A^{\star}x,y\rangle{}_X+\langle{}x,A^{\star}y\rangle{}_X = \langle{}\Gamma_1x,\Gamma_1x\rangle_H+\langle{}\Gamma_2x,\Gamma_2x\rangle_H=\|\Gamma_1x\|_H^2+\|\Gamma_2x\|_H^2
$$
where the left-hand side is zero since $A^{\star}x=-Ax$. Thus, $\Gamma_1x=\Gamma_2x=0$ follows.

\smallskip

\textquotedblleft{}$\supseteq$\textquotedblright{} Let $x\in D(-A^{\star})$ be given with $\Gamma_1x=\Gamma_2x=0$. Then the right-hand side of (BT-1) is zero for any $y\in D(-A^{\star})$ and we obtain by conjugation
$$
\langle{}A^{\star}y,x\rangle_X=-\langle{}y,A^{\star}x\rangle_X
$$
for all $y\in D(-A^{\star})$. Thus, the map $D(-A^{\star})\rightarrow\mathbb{C}$, $y\mapsto\langle{}A^{\star}y,x\rangle{}_X$ is continuous and consequently $x\in D((-A^{\star})^{\star})$. As $A$ is closed, \cite[X.1.6]{C} yields $A^{\star\star}=A$ and thus $x\in D((-A^{\star})^{\star})=D(A^{\star\star})=D(A)$ follows.

\smallskip

(ii) Let $x\in D(-A^{\star})$. We use (BT-1) to compute
$$
2\Re\langle{}A^{\star}x,x\rangle{}_X=\langle{}A^{\star}x,x\rangle{}_X+\langle{}x,A^{\star}x\rangle{}_X=\langle{}\Gamma_1x,\Gamma_2x\rangle_H+\langle{}\Gamma_2x,\Gamma_1x\rangle_H=2\Re\langle{}\Gamma_1x,\Gamma_2x\rangle_H
$$
which shows the equation in (ii).
\end{proof}

\smallskip

The next lemma contains a first characterization of dissipative extensions via an associated boundary triplet in terms of an estimate.

\smallskip

\begin{lem}\label{LEM-1} Let $B$ be an operator with $A\subseteq B\subseteq -A^{\star}$. Then $B$ is dissipative if and only if $\|\Gamma_1x\red{-}\Gamma_2x\|_H\leqslant\|\Gamma_1x\red{+}\Gamma_2x\|_H$ holds for all $x\in D(B)$.
\end{lem}
\begin{proof} For $x\in D(-A^{\star})$ and $\epsilon=\pm1$ we compute
\begin{equation*}
\begin{array}{r@{}l}\vspace{3pt}
\|\Gamma_1x+\epsilon\Gamma_2x\|_H^2 &{}\;=\; \langle{}\Gamma_1x+\epsilon\Gamma_2x, \Gamma_1x+\epsilon\Gamma_2x\rangle{}_H \\\vspace{5pt}
&{}\;=\; \langle{}\Gamma_1x,\Gamma_1x\rangle{}_H+\langle{}\Gamma_1x,\epsilon\Gamma_2x\rangle{}_H+\langle{}\epsilon\Gamma_2x,\Gamma_1x\rangle{}_H+\langle{}\epsilon\Gamma_2x,\epsilon\Gamma_2x\rangle{}_H\\
&{}\;=\; \|\Gamma_1x\|_H^2+\|\Gamma_2x\|_H^2+\epsilon\,2\Re\langle{}\Gamma_1x,\Gamma_2x\rangle{}_H.
\end{array}
\end{equation*}
Subtracting the two equations for $\epsilon=\pm1$, using Lemma \ref{LEM-0}(ii) and $B\subseteq-A^{\star}$ we obtain that
$$
\|\Gamma_1x\red{-}\Gamma_2x\|_H^2-\|\Gamma_1x\red{+}\Gamma_2x\|_H^2=-4\Re\langle{}\Gamma_1x,\Gamma_2x\rangle{}_H=4\Re\langle{}Bx,x\rangle{}_X
$$
holds for every $x\in D(B)$. The operator $B$ is dissipative if and only if the above is less or equal zero for every $x\in D(B)$. And this is equivalent to $\|\Gamma_1x\red{-}\Gamma_2x\|_H\leqslant\|\Gamma_1x\red{+}\Gamma_2x\|_H$ for every $x\in D(B)$.
\end{proof}

\smallskip

Our aim is now to parametrize the estimate in Lemma \ref{LEM-1} by a contractive operator. For a dissipative extension $B$ of $A$ we therefore define
\begin{equation}\label{PHI}
\begin{array}{r@{}l}\vspace{3pt}
\Phi(B)(\Gamma_1x\red{+}\Gamma_2x)&{}\;=\;\Gamma_1x\red{-}\Gamma_2x \\
D(\Phi(B))&{}\;=\;\{\Gamma_1x\red{+}\Gamma_2x\:;\:x\in D(B)\}
\end{array}
\end{equation}
which is a contractive operator in $H$ by the following lemma. Here we call $K\colon D(K)\rightarrow H$ \textit{contractive} if $\|Kx\|_H\leqslant\|x\|_H$ holds for all $x\in D(K)$. The term \textit{contraction} is reserved for those contractive $K$ which are defined on the whole space.

\smallskip

\begin{lem}\label{LEM-x} The operator $\Phi(B)\colon D(\Phi(B))\rightarrow H$ is well-defined and contractive.
\end{lem}
\begin{proof}  Let $x$, $y\in D(B)$ with $\Gamma_1x\red{+}\Gamma_2x=\Gamma_1y\red{+}\Gamma_2y$ be given. Put $z=x-y$. Then $\Gamma_1z=-\Gamma_2z$ holds. Since Proposition \ref{PROP-1} implies $B\subseteq-A^{\star}$ we can use Lemma \ref{LEM-0}(ii) and $B\subseteq-A^{\star}$ to compute 
$$
\Re\langle{}Bz,z\rangle_X = -\Re\langle{}\Gamma_1z,\Gamma_2z\rangle_H = \Re\langle{}\Gamma_1z,\Gamma_1z\rangle_H = \|\Gamma_1z\|_H^2.
$$
The latter is greater or equal to zero in view of the norm and less or equal to zero since $B$ is dissipative. Thus, $-\Gamma_2z=\Gamma_1z=0$ and therefore $\Gamma_1x=\Gamma_1y$ and $\Gamma_2x=\Gamma_2y$. This shows that $\Phi(B)$ is well defined. It is then clear that $\Phi(B)$ is linear.

\smallskip

To show the estimate, let $y\in D(\Phi(B))$ be given. By definition there is $x\in D(B)$ such that $y=\Gamma_1x\red{+}\Gamma_2x$ holds. The dissipativity of $B$ yields by Lemma \ref{LEM-1} the estimate
$$
\|\Phi(B)y\|_H=\|\Phi(B)(\Gamma_1x\red{+}\Gamma_2x)\|_H=\|\Gamma_1x\red{-}\Gamma_2x\|_H\leqslant\|\Gamma_1x\red{+}\Gamma_2x\|_H=\|y\|_H
$$
which shows that $\Phi(B)$ is contractive.
\end{proof}

\smallskip

The above yields a map $\Phi$ which assigns to a given dissipative extension a contractive operator. Next we consider a contractive operator $K\colon D(K)\rightarrow H$ and define 
\begin{equation}\label{PSI}
\begin{array}{r@{}l}\vspace{3pt}
\Psi(K)&{}\;=\;-A^{\star}|_{D(\Psi(K))}\\
D(\Psi(K))&{}\;=\;\bigl\{x\in D(-A^{\star})\:;\:\Gamma_1x\red{+}\Gamma_2x\in D(K) \,\text{ and } \,K(\Gamma_1x\red{+}\Gamma_2x)=\Gamma_1x\red{-}\Gamma_2x\bigr\}
\end{array}
\end{equation}
which is a dissipative operator in $X$ that extends $A$ by the following lemma.

\smallskip

\begin{lem}\label{LEM-2} The operator $\Psi(K)\colon D(\Psi(K))\rightarrow X$ is a dissipative extension of $A$.
\end{lem}
\begin{proof} We have $\Psi(K)\subseteq-A^{\star}$ by definition. Let $x\in D(A)$, i.e., $\Gamma_1x=\Gamma_2x=0$ by Lemma \ref{LEM-0}(i) and thus $x\in D(\Psi(K))$ holds. Therefore we have $A\subseteq\Psi(K)\subseteq-A^{\star}$.

\smallskip

Let $x\in D(\Psi(K))$. Then $K(\Gamma_1x\red{+}\Gamma_2x)=\Gamma_1x\red{-}\Gamma_2x$ holds. Since $K$ is contractive it \red{follows in particular that} $\|\Gamma_1x\red{-}\Gamma_2x\|_H\leqslant\|\Gamma_1x\red{+}\Gamma_2x\|_H$. Thus, $\Psi(K)$ is dissipative by Lemma \ref{LEM-1}.
\end{proof}

\smallskip

We showed that the map $\Psi$ assigns to a given contractive operator $K$ a dissipative extension. Below we investigate the relation between the maps $\Phi$ and $\Psi$.

\smallskip

\begin{prop}\label{PROP-2} (i) We have $(\Psi\circ\Phi)(B)=B$ for every dissipative extension $B$ of $A$.
\begin{compactitem}\vspace{3pt}
\item[(ii)] We have $(\Phi\circ\Psi)(K)=K$ for every contractive $K$ in $H$. 
\end{compactitem}
\end{prop}
\begin{proof} (i) Let $B$ be a dissipative extension of $A$. Then $(\Psi\circ\Phi)(B)=-A^{\star}|_{D((\Psi\circ\Phi)(B))}$ and since $B\subseteq-A^{\star}$ by Proposition \ref{PROP-1} it is enough to show that $D((\Psi\circ\Phi)(B))=D(B)$ holds. We have
$$
D((\Psi\circ\Phi)(B)) = \big\{x\in D(-A^{\star})\:;\:\Gamma_1x\red{+}\Gamma_2x\in D(\Phi(B)) \, \text{ and } \, \Phi(B)(\Gamma_1x\red{+}\Gamma_2x)=\Gamma_1x\red{-}\Gamma_2x\big\}.
$$
In view of \eqref{PHI}\red{,} the second condition holds provided that $x\in D(B)$. The first condition holds if and only if $x\in D(B)$. Therefore the above set equals $D(B)$ and we are done.

\smallskip

(ii) Let $K\colon D(K)\rightarrow H$ be contractive. We have
\begin{eqnarray*}
D((\Phi\circ\Psi)(K)) & = & \bigl\{\Gamma_1x\red{+}\Gamma_2x\:;\: x\in D(\Psi(K))\bigr\}\\
& = & \bigl\{\Gamma_1x\red{+}\Gamma_2x\:;\:x\in D(-A^{\star}),\:\Gamma_1x\red{+}\Gamma_2x\in D(K) \,\text{ and }\, K(\Gamma_1x\red{+}\Gamma_2x)=\Gamma_1x\red{-}\Gamma_2x\bigr\}
\end{eqnarray*}
and clearly $D((\Phi\circ\Psi)(K))\subseteq D(K)$ holds. To show the equality let $y\in D(K)$ be given. By (BT-2) we find $x\in D(-A^{\star})$ such that
$$
\Gamma_1x={\textstyle\frac{1}{2}}(y+Ky) \; \text{ and } \; \Gamma_2x={\textstyle\frac{1}{2}}(\red{y-Ky}).
$$
We compute $\Gamma_1x\red{+}\Gamma_2x=y\in D(K)$ and $\Gamma_1x\red{-}\Gamma_2x=Ky$. Thus, $K(\Gamma_1x\red{+}\Gamma_2x)=\Gamma_1x\red{-}\Gamma_2x$ and we have $y\in D((\Phi\circ\Psi)(K))$. Let now $\Gamma_1x\red{+}\Gamma_2x$ be an element of $D((\Phi\circ\Psi)(K))$. Then we have $K(\Gamma_1x\red{+}\Gamma_2x)=\Gamma_1x\red{-}\Gamma_2x=((\Phi\circ\Psi)(K))(\Gamma_1x\red{+}\Gamma_2x)$ which shows $K=(\Phi\circ\Psi)(K)$.
\end{proof}

By Proposition \ref{PROP-2} we get a one-to-one correspondence between the dissipative extensions of $A$ on $X$ and the contractive operators in $H$
\begin{equation}\label{BIG-DK}
\bigg\{\hspace{-3pt}\begin{array}{c}
B\colon D(B)\subseteq X\rightarrow X\\
A\subseteq B \; \& \; B \text{ dissipative}
\end{array}\hspace{-3pt}\bigg\}
\begin{array}{c}\vspace{-3pt}
\stackrel{\Phi}{\longrightarrow}\\\vspace{-5pt}
\stackrel{\textstyle\longleftarrow}{\scriptstyle\Psi}\vspace{6pt}
\end{array}
\bigg\{\hspace{-3pt}\begin{array}{c}
K\colon D(K)\subseteq H\rightarrow H\\
K \text{ contractive}
\end{array}\hspace{-3pt}\bigg\}
\end{equation}
where the maps $\Phi$ and $\Psi$ are defined as above.

\smallskip

Theorem \ref{MAIN} states that if $K$ is everywhere defined, then we recover the whole of $K$ and that the operators we get in $X$ are then precisely those which are maximal dissipative. Using the maps $\Phi$ and $\Psi$ on the new domains
\begin{equation}\label{DK}
\begin{array}{r@{}l}\vspace{3pt}
\mathcal{D}&{}\;=\;\big\{B\colon D(B)\rightarrow X\:;\: B \text{ is a maximal dissipative extension of } A\big\}\\
\mathcal{K}&{}\;=\;\big\{K\in L(H) \:;\:K \text{ is a contraction on }H\bigr\}
\end{array}
\end{equation}
the following proposition together with the Lumer-Phillips theorem, see Section \ref{SEC-3}, implies immediately the statement of Theorem \ref{MAIN}. \red{We emphasize at this point that there is a priori more than one maximal dissipative extension. In fact, Theorem \ref{MAIN} shows that there are always infinitely many provided that $H\not=\{0\}$.}

\smallskip

\begin{prop}\label{PROP-3-NEU} The map $\Phi\colon\mathcal{D}\rightarrow\mathcal{K}$ is bijective with inverse $\Psi\colon\mathcal{K}\rightarrow\mathcal{D}$.
\end{prop}
\begin{proof} The proof is divided into several steps.

\smallskip

(i) If $K\in\mathcal{K}$ and $\Psi(K)\subseteq{}B$ is dissipative, then we have $\Phi(B)=\Phi(\Psi(K))$.

\smallskip

Indeed, we observe
$$
D(\Phi(\Psi(K)))=\bigl\{\Gamma_1x\red{+}\Gamma_2x\:;\:x\in D(\Psi(K))\bigr\} \; \text{ and } \; \Phi(\Psi(K))(\Gamma_1x\red{+}\Gamma_2x)=\Gamma_1x\red{-}\Gamma_2x
$$
and by Proposition \ref{PROP-2} we have $\Phi(\Psi(K))=K$. Since $\Psi(K)\subseteq B$ holds by assumption, it \red{follows that} $\Phi(\Psi(K))\subseteq\Phi(B)$. As $D(\Phi(\Psi(K)))=D(K)=H$ is the whole space, we obtain $\Phi(B)=\Phi(\Psi(K))$.

\smallskip

(ii) The map $\Psi\colon\mathcal{K}\rightarrow\mathcal{D}$ is well-defined.

\smallskip

Let $K\in\mathcal{K}$ be given. Lemma \ref{LEM-2} yields that $\Psi(K)$ is dissipative; we have to show that it is maximal dissipative. Assume that $B$ is a dissipative extension of $\Psi(K)$, i.e., $\Psi(K)\subseteq B$. We claim $\Psi(K)=B$. Let $x\in D(B)$. That is, $\Gamma_1x\red{+}\Gamma_2x\in D(\Phi(B))$. By (i) and Proposition \ref{PROP-2} we have $\Phi(B)=K$ and it \red{follows that} $\Gamma_1x\red{+}\Gamma_2x\in D(K)$. Moreover, (i) implies $K=\Phi(\Psi(K))$ and thus
$$
\Gamma_1x\red{-}\Gamma_2x=\Phi(\Psi(K))(\Gamma_1x\red{+}\Gamma_2x)=K(\Gamma_1x\red{+}\Gamma_2x).
$$
In view of \eqref{PSI} it \red{follows that} $x\in D(\Psi(K))$.

\smallskip

(iii) If $B\in\mathcal{D}$ and $\Phi(B)\subseteq{}K$ with $K\in\mathcal{K}$, then we have $\Phi(B)=\Phi(\Psi(K))=K$.

\smallskip

We claim $B\subseteq\Psi(K)$. Let $x\in D(B)$. That is, $\Gamma_1x\red{+}\Gamma_2x\in D(\Phi(B))$. Since $\Phi(B)\subseteq K$ holds by assumption, we get by \eqref{PHI} that
$$
\Gamma_1x\red{-}\Gamma_2x=\Phi(B)(\Gamma_1x\red{+}\Gamma_2x)=K(\Gamma_1x\red{+}\Gamma_2x).
$$
Thus, $x\in D(\Psi(K))$ and $\Psi(K)x=-A^{\star}x=Bx$ holds by Proposition \ref{PROP-1}. Since $B$ is maximal dissipative and $\Psi(K)$ is dissipative by Lemma \ref{LEM-2} it \red{follows that} $B=\Psi(K)$. But now the conclusion follows from (i).

\smallskip
(iv) The map $\Phi\colon\mathcal{D}\rightarrow\mathcal{K}$ is well-defined.

\smallskip

Let $B\in\mathcal{D}$ be given. We show by contradiction that $D(\Phi(B))=H$ holds. Assume that this is not the case. By Lemma \ref{LEM-x} and the Hahn-Banach Theorem we can extend $\Phi(B)$ to an operator $K\in L(H)$ with $\|K\|_{L(H)}\leqslant1$, i.e., $K\in\mathcal{K}$. By (iii) it \red{follows that} $\Phi(B)=K$ but by our assumptions $\Phi(B)\subset K$ is a proper subset. The contradiction shows that $\Phi(B)\in\mathcal{K}$ holds.

\smallskip

We established that $\Phi\colon\mathcal{D}\rightarrow\mathcal{K}$ and $\Psi\colon\mathcal{K}\rightarrow\mathcal{D}$ are both well-defined. Thus it follows from Proposition \ref{PROP-2} that $\Psi\circ\Phi=\id_{\mathcal{D}}$ and $\Phi\circ\Psi=\id_{\mathcal{K}}$ are valid, which finishes the proof.
\end{proof}

\bigskip


\smallskip

\section{Existence of Boundary Triplets}\label{SEC-5}

In this section we characterize those densely defined, closed and skew-symmetric operators $A$ for which a boundary triplet exists. We note that the applications in Section \ref{SEC-6} are independent of the following. The reader can thus proceed directly with the port-Hamiltonian systems, if he or she likes to.

\smallskip

\begin{thm}\label{EX-THM} Let $A$ be a densely defined, closed and skew-symmetric operator. Then there exists a boundary triplet $(H,\Gamma_1,\Gamma_2)$ for $A$ if and only if $\dim\ker(1-A^{\star})=\dim\ker(1+A^{\star})$ holds.
\end{thm}

\smallskip

Note that the spaces $\ker(1-A^{\star})$ and $\ker(1+A^{\star})$ are closed, cf.~the proof of \red{Lemma \ref{LEM}(vi)}. In the situation of symmetric operators the dimensions of the spaces corresponding to $\ker(1-A^{\star})$ and $\ker(1+A^{\star})$ are usually called the \textit{deficiency indices} of $A$. Below, we first prove that their equality is sufficient for the existence of a boundary triplet. This follows immediately from the decomposition established in Lemma \ref{DECOMP-LEM} by computation.

\smallskip

\begin{prop}\label{PROP-SUFF} Let $A$ be a densely defined, closed and skew-symmetric operator. If 
$\dim\ker(1-A^{\star})=\dim\ker(1+A^{\star})$ holds, then there exits a boundary triplet $(H,\Gamma_1,\Gamma_2)$ for $A$.
\end{prop}
\begin{proof} By Lemma \ref{DECOMP-LEM} we have $D(-A^{\star})=D(A)\oplus{}\ker(1-A^{\star})\oplus{} \ker(1+A^{\star})$ and we denote by
$$
P_1\colon D(-A^{\star})\rightarrow\ker(1-A^{\star}) \;\; \text{ and } \;\; P_2\colon D(-A^{\star})\rightarrow\ker(1+A^{\star})
$$
the corresponding projections. Since $\dim\ker(1-A^{\star})=\dim\ker(1+A^{\star})$, we find an isometric isomorphism $\varphi\colon\ker(1+A^{\star})\rightarrow\ker(1-A^{\star})$. We put $H=\ker(1-A^{\star})$ endowed with the inner product induced by $X$, define
$$
\Gamma_1=P_1+\varphi{}\circ P_2 \;\; \text{ and } \;\; \Gamma_2=P_1-\varphi{}\circ P_2
$$
and check the conditions in Definition \ref{dfn-BT}. Let $x, y\in D(-A^{\star})$ be given. According to the above we have $x=x_0+P_1x+P_2x$ and $y=y_0+P_1y+P_2y$ with $x_0$, $y_0\in D(A)$ and the equalities $A^{\star}(P_1x)=P_1x$ and $A^{\star}(P_2x)=-P_2x$. Using the latter, the fact that $A$ is skew-symmetric, and that $\varphi$ is isometric, a long but straightforward computation shows
\begin{eqnarray*}
\langle{}A^{\star}x,y\rangle{}_X+\langle{}x,A^{\star}y\rangle{}_X &=& \langle{}P_1x,P_1y\rangle{}_X-\langle{}P_2x,P_2y\rangle{}_X+\langle{}P_1x,P_1y\rangle{}_X-\langle{}P_2x,P_2y\rangle{}_X\\
&=&\langle{}P_1x+\varphi{}(P_2x),P_1y-\varphi{}(P_2y)\rangle{}_H+\langle{}P_1x-\varphi{}(P_2x),P_1y+\varphi{}(P_2y)\rangle{}_H \\
&=&\langle{}\Gamma_1x,\Gamma_2y\rangle{}_H+\langle{}\Gamma_2x,\Gamma_1y\rangle{}_H
\end{eqnarray*}
which proves (BT-1). Let $y_1$, $y_2\in H$ be given. We put $x=x_0+x_1+x_2\in D(-A^{\star})$ with $x_0=0\in D(A)$, $x_1=\frac{1}{2}(y_1+y_2)\in H=\ker(1-A^{\star})$ and $x_2=\frac{1}{2}\varphi^{-1}(y_1-y_2)\in\ker(1+A^{\star})$. We get $\Gamma_1x=x_1+\varphi{}(x_2)=\frac{1}{2}(y_1+y_2)+\frac{1}{2}(y_1-y_2)=y_1$ and similarly $\Gamma_2x=x_1-\varphi{}(x_2)=\frac{1}{2}(y_1+y_2)-\frac{1}{2}(y_1-y_2)=y_2$ which finishes the proof of (BT-2).
\end{proof}

\smallskip

The necessity part of Theorem \ref{EX-THM} uses the main result of Section \ref{SEC-4} in a special case and needs the following preparation.

\smallskip

\begin{lem}\label{LEM-K-UNITARY} Let $A$ be a densely defined, closed and skew-symmetric operator and let $(H,\Gamma_1,\Gamma_2)$ be a boundary triplet for $A$. We consider $K=-\id_H$. Then the operator $B=\Psi(K)$ satisfies $D(B)=\ker\Gamma_1$ and $B=-B^{\star}$.
\end{lem}
\begin{proof} We first show that $D(B)=\ker\Gamma_1$ is true.

\smallskip

\textquotedblleft{}$\subseteq$\textquotedblright{} Let $x\in D(B)$. That is 
$$
\Gamma_1x\red{-}\Gamma_2x=K(\Gamma_1x\red{+}\Gamma_2x)=K(\Gamma_1x)\red{+}K(\Gamma_2x)=-\Gamma_1x\red{-}\Gamma_2x.
$$
Consequently, $\Gamma_1x=-\Gamma_1x$ which implies $x\in\ker\Gamma_1$.

\smallskip

\textquotedblleft{}$\supseteq$\textquotedblright{} Let $x\in\ker\Gamma_1$, i.e., $\Gamma_1x=0$ and thus $K(\Gamma_1x\red{+}\Gamma_2x)=K(\Gamma_2x) = 0\red{+}(-\Gamma_2x)=\Gamma_1x\red{-}\Gamma_2x$, which means $x\in D(B)$.

\smallskip

Now we show $B=-B^{\star}$. By \red{Proposition \ref{PROP-3-NEU}} we know that $B$ is a maximal dissipative extension of $A$ and thus \red{Lemma \ref{LEM-3}} gives $A\subseteq B\subseteq-A^{\star}$. Therefore $(-A^{\star})^{\star}\subseteq B^{\star}\subseteq A^{\star}$ holds. Since $A$ is closed we may use \cite[X.1.6]{C} to get that $(-A^{\star})^{\star}=A$ holds. This implies finally $A\subseteq -B^{\star}\subseteq -A^{\star}$.

\smallskip

It remains to show that $D(-B^{\star})=D(B)$ holds. We know already that both sets are contained in $D(-A^{\star})$ and thus we fix $y\in D(-A^{\star})$. In view of $B^{\star}\subseteq A^{\star}$ we get that $y\in D(-B^{\star})$ holds if and only if $\langle{}Bx,y\rangle{}_X = \langle{}x,A^{\star}y\rangle{}_X$ is true for every $x\in D(B)$. \red{Using $B\subseteq-A^{\star}$ we can replace} the equality in the statement with
$\langle{}-A^{\star}x,y\rangle{}_X = \langle{}x,A^{\star}y\rangle{}_X$. With (BT-1) the latter means exactly $\langle{}\Gamma_1x,\Gamma_2y\rangle{}_H+\langle{}\Gamma_2x,\Gamma_1y\rangle{}_H=0$. We thus have
\begin{eqnarray*}
D(-B^{\star}) &=& \bigl\{y\in D(-A^{\star})\:;\:\forall\:x\in D(B)\colon\langle{}\Gamma_1x,\Gamma_2y\rangle{}_H+\langle{}\Gamma_2x,\Gamma_1y\rangle{}_H=0\bigr\}\\
& = & \bigl\{y\in D(-A^{\star})\:;\:\forall\:x\in D(B)\colon\langle{}\Gamma_1x\red{+}\Gamma_2x,\Gamma_1y\rangle{}_H=0\bigr\}\\
& = &  \bigl\{y\in D(-A^{\star})\:;\:\forall\:h\in H\colon\langle{}h,\Gamma_1y\rangle{}_H=0\bigr\}\\
& = & \ker\Gamma_1,
\end{eqnarray*}
where we used $D(B)=\ker\Gamma_1$ and the fact that $H=D(\Psi(B))=D(K)$, which follows from Proposition \ref{PROP-2}. Thus we have $D(B)=\ker\Gamma_1=D(-B^{\star})$.
\end{proof}

\smallskip

We note that in the literature an operator $B$ with $B=-B^{\star}$ is said to be \textit{skew-adjoint}.

\smallskip

\begin{lem}\label{LEM-xxx} Let $A$ be a densely defined, closed and skew-symmetric operator. Let $B\colon D(B)\rightarrow X$ be a maximal dissipative extension. Assume that $B=-B^{\star}$ holds. Then $D(-A^{\star})=D(B)\oplus{}\ker(1+\epsilon{}A^{\star})$ holds for $\epsilon=\pm1$. \red{The direct sum is not necessarily orthogonal.}
\end{lem}
\begin{proof} Let $\epsilon=\pm1$. We have $\ker(1+\epsilon{}A^{\star})\subseteq D(-A^{\star})$ and $D(B)\subseteq D(-A^{\star})$. We observe that $\Re\langle{}Bx,x\rangle{}_X=0$ holds for every $x\in D(B)$ since $B=-B^{\star}$.

\smallskip

Let $x\in D(B)\cap\ker(1+\epsilon{}A^{\star})$, i.e., $A^{\star}x=-\epsilon{}x$ and $Bx=-A^{\star}x=\epsilon{}x$. We get
$$
0=\Re\langle{}Bx,x\rangle{}_X=\epsilon\Re\langle{}x,x\rangle{}_X=\epsilon\|x\|^2_X
$$
and therefore $x=0$. Thus we showed that the sum $D(B)+\ker(1+\epsilon{}A^{\star})$ is direct.

\smallskip

It remains to prove $D(-A^{\star})=D(B)+\ker(1+A^{\star})$.

\smallskip

\textquotedblleft{}$\supseteq$\textquotedblright{} Since $B$ is dissipative we have $B\subseteq -A^{\star}$ by Proposition \ref{PROP-1}.

\smallskip

\textquotedblleft{}$\subseteq$\textquotedblright{} We first show that $1-\epsilon{}B\colon D(B)\rightarrow X$ is bijective. Let $x\in D(B)$ be such that $(1-\epsilon{}B)x=0$, i.e., $Bx=\epsilon{}x$. Then $0=\Re\langle{}Bx,x\rangle{}_X=\epsilon{}\|x\|_X^2$ implies $x=0$. \red{Therefore $1-\epsilon{}B$ is injective.} From Lemma \ref{SIMPLE-DECOMP} we get $X=\ran(1-\epsilon{}B)\oplus{}\ker(1-\epsilon{}B^{\star})$ where \red{$\ker(1-\epsilon{}B^{\star})=\ker(1+\epsilon{}B)=\{0\}$}. Thus, $1-\epsilon{}B\colon D(B)\rightarrow X$ is also surjective.

\smallskip

Now we take $x\in D(-A^{\star})$ and decompose $(1+\epsilon{}A^{\star})x$ according to
$$
X=\ran(1-\epsilon{}A)\oplus{}\ker(1-\epsilon{}A^{\star}),
$$
cf.~Lemma \ref{SIMPLE-DECOMP}. Accordingly, we find $x_0\in D(A)$ and $y_1\in\ker(1-\epsilon{}A^{\star})$ such that $(1+\epsilon{}A^{\star})x=(1-\epsilon{}A)x_0+y_1$. We put $y_2=x-x_0-(1-\epsilon{}B)^{-1}y_1$ and compute
\begin{eqnarray*}
(1+\epsilon{}A^{\star})y_2 & = & (1+\epsilon{}A^{\star})x - (1+\epsilon{}A^{\star})x_0 - (1+\epsilon{}A^{\star})(1-\epsilon{}B)^{-1}y_1\\
& = & (1+\epsilon{}A^{\star})x - (1-\epsilon{}A)x_0 - (1-\epsilon{}B)(1-\epsilon{}B)^{-1}y_1\\
& = & y_1 - y_1 \; = \; 0
\end{eqnarray*}
that is $y_2\in\ker(1+\epsilon{}A^{\star})$. On the other hand, $x_0\in D(A)\subseteq D(B)$ and $(1-\epsilon{}B)^{-1}y_1\in D(B)$ imply $x_0+(1-\epsilon{}B)^{-1}y_1\in D(B)$ and thus
$$
x=x_0+(1-\epsilon{}B)^{-1}y_1+y_2\in D(B)+\ker(1+\epsilon{}A^{\star})
$$
as desired.
\end{proof}

\smallskip

The last ingredient for the proof of the necessity part is the following.

\smallskip

\begin{lem}\label{LEM-CTNS} Let $A$ be a densely defined, closed and skew-symmetric operator and let $(H,\Gamma_1,\Gamma_2)$ be a boundary triplet for $A$. Then the map $\Gamma_1\colon D(-A^{\star})\rightarrow H$ is continuous when $D(-A^{\star})$ is endowed with the graph norm $\|\cdot\|_{D(-A^{\star})}$ defined by $\|x\|_{D(-A^{\star})}:=\|x\|_X+\|-A^{\star}x\|_X$ for $x\in D(-A^{\star})$.
\end{lem}
\begin{proof} We show that the map
$$
\Gamma\colon(D(-A^{\star}),\|\cdot\|_{D(-A^{\star})})\rightarrow H\times{}H,\,x\mapsto (\Gamma_1x,\Gamma_2x)
$$
is continuous where we use the product topology on $H\times{}H$. We apply the version \cite[III.12.7]{C} of the closed graph theorem to show the latter. Let $(x_n)_{n\in\mathbb{N}}\subseteq D(-A^{\star})$ be given with $x_n\rightarrow 0\in D(-A^{\star})$ with respect to $\|\cdot\|_{D(-A^{\star})}$ and $\Gamma_1x_n\rightarrow y_1$, $\Gamma_2x_n\rightarrow y_2$ for $n\rightarrow\infty$ with $y_1$, $y_2\in H$. Then, for arbitrary $y\in D(-A^{\star})$, we can compute
\begin{eqnarray*}
\langle{}y_1,\Gamma_2y\rangle{}_H + \langle{}y_2,\Gamma_1y\rangle{}_H &=& \lim_{n\rightarrow\infty}\bigl[\langle{}\Gamma_1x_n,\Gamma_2y\rangle{}_H + \langle{}\Gamma_2x_n,\Gamma_1y\rangle{}_H\bigr]\\
&=& \lim_{n\rightarrow\infty}\bigl[\langle{}A^{\star}x_n,y\rangle{}_X + \langle{}x_n,A^{\star}y\rangle{}_X\bigr]\\
&=& \langle{}\lim_{n\rightarrow\infty}A^{\star}x_n,y\rangle{}_X + \langle{}\lim_{n\rightarrow\infty}x_n,A^{\star}y\rangle{}_X\\
&=& 0
\end{eqnarray*}
since $A^{\star}\colon (D(-A^{\star}),\|\cdot\|_{D(-A^{\star})})\rightarrow X$ is continuous. According to (BT-2) we select $y\in D(-A^{\star})$ such that $\Gamma_1y=0$ and $\Gamma_2y=y_1$. Then we get from the above
$$
0=\langle{}y_1,\Gamma_2y\rangle{}_H + \langle{}y_2,\Gamma_1y\rangle{}_H=\|y_1\|^2_X 
$$
and thus $y_1=0$. Analogously, it \red{follows that} $y_2=0$ if we select $y\in D(-A^{\star})$ such that $\Gamma_1y=y_2$, $\Gamma_2y=0$ and use the formula above. Thus, $\Gamma$ is continuous and \red{hence} also $\Gamma_1$ as this map equals the composition of $\Gamma$ and the projection on the first component of $H\times{}H$.
\end{proof}

\smallskip

The necessity part of Theorem \ref{EX-THM} is exactly the statement below.

\smallskip

\begin{prop} Let $A$ be a densely defined, closed and skew-symmetric operator and let $(H,\Gamma_1,\Gamma_2)$ be a boundary triplet for $A$. Then $\dim\ker(1-A^{\star})=\dim\ker(1+A^{\star})$ holds.
\end{prop}
\begin{proof} We put $K=-\id_H$. By Lemma \ref{LEM-K-UNITARY}\red{,} $B=\Psi(K)$ satisfies $B=-B^{\star}$ and $D(B)=\ker\Gamma_1$. We fix $\epsilon=\pm1$ and use Lemma \ref{LEM-xxx} to derive
$$
D(-A^{\star})=D(B)\oplus\ker(1+\red{\epsilon}A^{\star})=\ker\Gamma_1\oplus\ker(1+\epsilon{}A^{\star}).
$$
By (BT-2) the map $\Gamma_1\colon\ker\Gamma_1\oplus\ker(1+\epsilon{}A^{\star})\rightarrow H$ is surjective and in view of the decomposition this implies immediately that $\Gamma_1|_{\ker(1+\epsilon{}A^{\star})}\colon\ker(1+\epsilon{}A^{\star})\rightarrow H$ is bijective.

\smallskip

Since $A^{\star}$ is closed, $\ker(1+\epsilon{}A^{\star})\subseteq X$ is closed and thus a Banach space w.r.t.~the norm $\|\cdot\|_X$. Since
$$
\|x\|_X\leqslant\|x\|_X+\|-A^{\star}x\|_X=\|x\|_X+\|\epsilon{}A^{\star}x\|_X=2\|x\|_X
$$
holds for $x\in\ker(1+\epsilon{}A^{\star})$, the norm $\|\cdot\|_X$ and the graph norm $\|\cdot\|_{D(-A^{\star})}$, see Lemma \ref{LEM-CTNS}, are equivalent on $\ker(1+\epsilon{}A^{\star})$. Lemma \ref{LEM-CTNS} thus shows that \red{$\Gamma_1|_{\ker(1+\epsilon{}A^{\star})}\colon\ker(1+\epsilon{}A^{\star})\rightarrow H$} is continuous. The open mapping theorem yields that the latter is an isomorphism.
\end{proof}

\bigskip


\section{Port-Hamiltonian Systems}\label{SEC-6}

\smallskip

In this section we apply the results of Section \ref{SEC-4} to port-Hamiltonian systems. We follow Le Gorrec, Zwart, Maschke \cite{GZM2005} but add some more details and are able to remove some technicalities due to our relation-free and throughout skew-symmetric approach above.

\smallskip

Let $d\geqslant1$ be a fixed integer and $a<b$ be real numbers. We consider the partial differential equation
$$
{\textstyle\frac{\partial x}{\partial t}}(t,\xi)=P_0(\mathcal{H}x)(t,\xi)+P_1{\textstyle\frac{\partial(\mathcal{H}x)}{\partial \xi}}(t,\xi)\:\text{ for }\:t\geqslant0\:\text{ and }\:\xi\in(a,b)
$$
where $P_0$, $P_1\in\mathbb{C}^{d\times d}$ with $P_0^{\star}=-P_0$ and $P_1^{\star}=P_1$ invertible. We assume that $\mathcal{H}\colon(a,b)\rightarrow\mathbb{C}^{d\times d}$ is measurable and that there exists $0<m\leqslant M$ \red{such that} for almost every $\xi\in(a,b)$ the matrix $\mathcal{H}(\xi)$ is self-adjoint and that we have
$$
m|\zeta|^2\leqslant\langle{}\mathcal{H}(\xi)\zeta,\zeta\rangle{}_{\mathbb{C}^d}\leqslant M|\zeta|^2
$$
for all $\zeta\in\mathbb{C}^d$ where $\langle{}\cdot,\cdot\rangle{}_{\mathbb{C}^d}$ denotes the inner product of the euclidean space $\mathbb{C}^d$. In the sequel, we write $L^2$ for the space $L^2(a,b;\mathbb{C}^d)$ with the standard inner product $\langle{}\cdot,\cdot\rangle{}_{L^2}$ which is linear in the first argument. We use the letter $\mathcal{H}$ for the operator $\mathcal{H}\colon L^2\rightarrow L^2$, $(\mathcal{H}x)(\xi)=\mathcal{H}(\xi)x(\xi)$ for $\xi\in(a,b)$, which is a self-adjoint operator.

\smallskip

\begin{lem}\label{H-LEM} Let $\mathcal{H}$ be as above.\vspace{3pt}
\begin{compactitem}
\item[(i)] The space $(X,\langle{}\cdot,\cdot\rangle{}_X)$ with $X=L^2$ and $\langle{}\cdot,\cdot\rangle{}_X=\langle{}\mathcal{H}\cdot,\cdot\rangle{}_{L^2}$ is a Hilbert space.\vspace{3pt}
\item[(ii)] The operator $\mathcal{H}\colon X\rightarrow L^2$ is an isomorphism of Banach spaces.
\end{compactitem}
\end{lem}
\begin{proof}(i) From our assumptions on $\mathcal{H}$ it follows that $\langle\cdot,\cdot\rangle{}_X$ is an inner product on $X$ and that $\|\cdot\|_{L^2}$ and $\|\cdot\|_X$ are equivalent. Therefore $(X,\|\cdot\|_X)$ is complete and thus $(X,\langle{}\cdot,\cdot\rangle{})$ is Hilbert. 

\smallskip

(ii) We first consider $\mathcal{H}\colon L^2\rightarrow L^2$ and observe that this operator is continuous. By our assumptions on $\mathcal{H}$ we have $\langle{}\mathcal{H}x,x\rangle{}_{L^2}\geqslant{}m\|x\|^2_{L^2}$ for every $x\in L^2$. This shows that $\mathcal{H}$ is injective and we can consider $\mathcal{H}^{-1}\colon\ran\mathcal{H}\rightarrow L^2$. For $y=\mathcal{H}x\in\ran\mathcal{H}$ with $x\in L^2$ we get
$$
\|\mathcal{H}^{-1}y\|_{L^2}=\|x\|_{L^2}\leqslant\|\mathcal{H}x\|_{L^2}=\|y\|_{L^2}
$$
by the Bunyakovsky-Cauchy-Schwarz inequality. Thus, $\mathcal{H}^{-1}$ is continuous and  
$\ran\mathcal{H}\subseteq{}L^2$ is closed. Let $y\in(\ran\mathcal{H})^{\perp}$, i.e., $\langle{}Ax,y\rangle{}=0$ for all $x\in L^2$. Putting $x=y$ shows together with our initial estimate, that only $y=0$ is possible. Thus, $\ran\mathcal{H}\subseteq L^2$ is dense by \cite[I.2.11]{C}. Consequently, $\ran\mathcal{H}=L^2$ and we proved that $\mathcal{H}\colon L^2\rightarrow L^2$ is an isomorphism of Banach spaces. The proof of (i) showed that $\id_{L^2}\colon X\rightarrow L^2$ is an isomorphism. A combination of both statements establishes that 
$\mathcal{H}\colon X\rightarrow L^2$ is an isomorphism. 
\end{proof}

\smallskip

We denote by $H^1$ the Sobolev space $H^1(a,b,\mathbb{C}^d)$ whose members are all absolutely continuous functions $x\in L^2$---in particular $x$ is almost everywhere differentiable---such that also $x'\in L^2$ holds. We define the operator $A_0\colon D(A_0)\rightarrow X$ via
\begin{equation*}
A_0x = P_0\mathcal{H}x + P_1(\mathcal{H}x)'\: \text{ for } \: x\in D(A_0)=\big\{x\in X\:;\:\mathcal{H}x\in H^1 \text{ and } x(a)=x(b)=0\big\}
\end{equation*}
and define $\Gamma_1$, $\Gamma_2\colon H^1\rightarrow\mathbb{C}^d$ via
\begin{equation*}
\Gamma_1y={\textstyle\frac{1}{\sqrt{2}}}P_1(y(b)-y(a)) \;\; \text{ and } \;\; \Gamma_2y={\textstyle\frac{1}{\sqrt{2}}}(y(b)+y(a))
\end{equation*}
for $y\in H^1$. Our aim for this section is to employ the method of boundary triplets that we established in Section \ref{SEC-4} to classify those extensions of $A_0$ which generate a $\Cnull$-semigroup.

\medskip

In view of the condition in Theorem \ref{PHS-THM} below we remind the reader at this point that for $M\in\mathbb{C}^{d\times{}d}$ the estimate $M\geqslant 0$ holds by definition if $\langle{}M\xi,\xi\rangle{}_{\mathbb{C}^d}\geqslant0$ for all $\xi\in\mathbb{C}^d$.

\smallskip

\begin{thm}\label{PHS-THM} If $W\in\mathbb{C}^{d\times{}2d}$ satisfies $\rk W=d$ and $W\Sigma{}W^{\star}\geqslant0$ with $\Sigma=\bigl[\begin{smallmatrix}0&I\\ I&0\end{smallmatrix} \bigr]$, then the operator $B_0\colon D(B_0)\rightarrow X$ defined by
\begin{equation*}
B_0x = P_0\mathcal{H}x + P_1(\mathcal{H}x)'\: \text{ for } \: x\in D(B_0)=\big\{x\in X\:;\:\mathcal{H}x\in H^1 \text{ and } W{\textstyle\columnvec{\Gamma_1\mathcal{H}x}{\Gamma_2\mathcal{H}x}}=0\big\}
\end{equation*}
extends the operator $A_0\colon D(A_0)\rightarrow X$ and generates a $\Cnull$-semigroup of contractions. Conversely, every operator $B_0\colon D(B_0)\rightarrow X$, that extends $A_0\colon D(A_0)\rightarrow X$ and generates a $\Cnull$-semigroup of contractions, is of the form above for some $W$ with the properties above.
\end{thm}

\smallskip

We use the following lemma in order to reduce the proof to the special case $\mathcal{H}=\id_{L^2}$.

\smallskip

\begin{lem}\label{REDUCTION} Let $\mathcal{H}$ and $X$ be as above. An operator $B\colon D(B)\rightarrow L^2$ is maximal dissipative in $L^2$ if and only if the operator $B_0\colon D(B_0)\rightarrow X$ defined by $B_0x=B\mathcal{H}x$ for $x\in D(B_0)=\{x\in X\:;\:\mathcal{H}x\in D(B)\}$ is maximal dissipative in $X$.
\end{lem}
\begin{proof} For $x\in D(B_0)$, i.e., $\mathcal{H}x\in D(B)$, we have $\langle{}B_0x,x\rangle{}_{X}=\langle{}\mathcal{H}B\mathcal{H}x,x\rangle{}_{L^2}=\langle{}B\mathcal{H}x,\mathcal{H}x\rangle{}_{L^2}$. Since the isomorphism $\mathcal{H}\colon X\rightarrow L^2$ maps $D(B_0)$ onto $D(B)$ we get that $B$ is densely defined if and only if $B_0$ is densely defined. Together with our first computation the latter shows that $B_0$ is dissipative in $X$ if and only if $B$ is dissipative in $L^2$.

\smallskip

If $C$ is an extension of $B$ then $C_0\colon D(C_0)\rightarrow X$, defined by $C_0x=C\mathcal{H}x$ for $x\in D(C_0)=\{x\in X\:;\:\mathcal{H}x\in D(C)\}$, is an extension of $B_0$. If $C_0$ is an extension of $B_0$, then $C\colon D(C)\rightarrow L^2$ defined by $Cx=C_0\mathcal{H}^{-1}x$ for $x\in D(C)=\{x\in L^2\:;\:\mathcal{H}^{-1}x\in D(C_0)\}$ is an extension of $B$.  The above assignments are inverse to each other and by the first part of the proof, the operator $C$ is dissipative if and only if $C_0$ is dissipative. We thus established a one-to-one correspondence between the dissipative extensions of $B$ and $B_0$ which finishes the proof.
\end{proof}

\smallskip

By the lemma above it is enough to prove Theorem \ref{PHS-THM} for the case $\mathcal{H}=\id_{L^2}$. To avoid confusion, we state this special case explicitly below. First, we fix the following notation for the rest of this section.

\smallskip

Let $A\colon D(A)\rightarrow L^2$ be given by
\begin{equation*}
Ax = P_0x + P_1x'\: \text{ for } \: x\in D(A)=\big\{x\in H^1\:;\: x(a)=x(b)=0\big\}
\end{equation*}
and fix $\Gamma_1$, $\Gamma_2\colon H^2\rightarrow \mathbb{C}^d$ with
\begin{equation*}
\Gamma_1y={\textstyle\frac{1}{\sqrt{2}}}P_1(y(b)-y(a)) \;\; \text{ and } \;\; \Gamma_2y={\textstyle\frac{1}{\sqrt{2}}}(y(b)+y(a))
\end{equation*}
for $y\in H^1$. Finally we put
\begin{equation*}
\mathcal{W} = \big\{W\in\mathbb{C}^{d\times2d}\:;\:\rk W=d \text{ and } W\Sigma W^{\star}\geqslant0\big\} \; \text{ with } \; \Sigma=\bigl[\begin{smallmatrix}0&I\\ I&0\end{smallmatrix} \bigr]
\end{equation*}
and have to prove the following.

\smallskip

\begin{prop}\label{PROP-PHS} (i) If $W\in\mathcal{W}$ then the operator $B\colon D(B)\rightarrow L^2$ given by
\begin{equation*}
Bx = P_0x + P_1x'\: \text{ for } \: x\in D(B)=\big\{x\in H^1\:;\:W{\textstyle\columnvec{\Gamma_1x}{\Gamma_2x}}=0\big\}
\end{equation*}
\phantom{a}\hspace{31pt}generates a $\Cnull$-semigroup of contractions.\vspace{3pt}
\begin{compactitem}
\item[(ii)] If $B\colon D(B)\rightarrow X$ extends $A\colon D(A)\rightarrow X$ and generates a $\Cnull$-semigroup of contractions, then there exists $W\in\mathcal{W}$ such that $B\colon D(B)\rightarrow L^2$ is of the form above.
\end{compactitem}
\end{prop}

\smallskip

We start with the following properties of the operator $A\colon D(A)\rightarrow L^2$ and the maps $\Gamma_1$ and $\Gamma_2$.

\smallskip

\begin{lem}\label{LEM-5} The operator $A\colon D(A)\rightarrow L^2$ is densely defined, closed and skew-symmetric. More precisely, $-A^{\star}\colon D(-A^{\star})\rightarrow X$ is given by $-A^{\star}x = P_0x+P_1x'$ for $x\in D(-A^{\star})=H^1$.
\end{lem}
\begin{proof} That $D(A)\subseteq L^2$ is dense can be seen by the same arguments used in \cite[Example X.1.11]{C}.

\smallskip

The multiplication operators $P_k\colon L^2\rightarrow L^2$, $[P_k(x)](\xi)=P_kx(\xi)$ for $\xi\in(a,b)$ and $k=0,1$ are continuous. The derivative $D\colon D(A)\rightarrow L^2$ \red{is closed}, see again \cite[Example X.1.11]{C}, and thus also $P_1D$ is closed. Finally, $A=P_0+P_1D$ is closed as the sum of a continuous and a closed operator.

\smallskip

We put $A_1=P_1D$ with $D(A_1)=D(A)$. We claim that $D(A_1^{\star})=H^1$. Let $y\in D(A_1^{\star})$, put $z=P_1^{-1}A_1^{\star}y$ and $Z(\cdot)=\int_a^{\cdot}z(\xi)d\xi$. For $x\in D(A_1)$ we compute
$$
{\textstyle\int_a^b}\langle{}x'(\xi),P_1y(\xi)\rangle{}_{\mathbb{C}^d}d\xi = \langle{}A_1x,y\rangle{}_{L^2} = \langle{}x,A_1^{\star}y\rangle{}_{L^2} = \langle{}x,P_1z\rangle{}_{L^2} =-{\textstyle\int_a^b}\langle{}x'(\xi),P_1Z(\xi)\rangle{}_{\mathbb{C}^d}d\xi
$$
where the last equation follows by component-wise integration by parts. We showed that $\langle{}x',P_1(y+Z)\rangle{}_{L^2}=0$ holds for all $x\in D(A_1)$. Consequently,
$$
\red{P_1(y+Z)\in\{x'\:;\:x\in H^1 \text{ and } x(a)=x(b)=0\}^{\perp}=\{c\colon[a,b]\rightarrow\mathbb{C}^d\:;\:c \text{ is constant}\}}
$$
and thus there exists $c\in\mathbb{C}^d$ such that $P_1(y(\xi)+Z(\xi))=c$ for all $\xi\in(a,b)$. Therefore, $y=P_1^{-1}c-Z$ is absolutely continuous and $y'=-z\in L^2$, i.e., $y\in H^1$ and in addition $A_1^{\star}y=P_1z=-P_1y'$. Let now $y\in H^1$. For $x\in D(A_1)$, in particular $x(a)=x(b)=0$, we compute
$$
\langle{}A_1x,y\rangle{}_{L^2} = {\textstyle\int_a^b}\langle{}P_1x'(\xi),y(\xi)\rangle{}_{\mathbb{C}^d}d\xi = -{\textstyle\int_a^b}\langle{}x(\xi),P_1y'(\xi)\rangle{}_{\mathbb{C}^d}d\xi = \langle{}x,-P_1y'\rangle{}_{L^2}
$$
which shows that the map $D(A_1)\rightarrow\mathbb{C}$, $x\mapsto\langle{}A_1x,y\rangle{}_{L^2}$ is continuous, i.e., $y\in D(A_1^{\star})$.

\smallskip

Since the multiplication operator $P_0\colon L^2\rightarrow L^2$ is continuous and the matrix $P_0$ satisfies $P_0^{\star}=-P_0$, it \red{follows that} $A^{\star}=(P_0+A_1)^{\star}=P_0^{\star}+A^{\star}$, i.e., $D(-A^{\star})=H^1$ and $A^{\star}x=-P_0x-P_1x'$. In particular,  $A$ is skew-symmetric.
\end{proof}

\smallskip

\begin{lem}\label{LEM-6} The triple $(\mathbb{C}^d,\Gamma_1,\Gamma_2)$ is a boundary triplet for $A$.
\end{lem}
\begin{proof} By Lemma \ref{LEM-5} the maps $\Gamma_1$, $\Gamma_2\colon D(-A^{\star})\rightarrow\mathbb{C}^d$ are well-defined. We have to check the two conditions in Definition \ref{dfn-BT}.

\smallskip

Let $x, y\in H^1$ be given. We compute
\begin{eqnarray*}
\langle{}A^{\star}x,y\rangle{}_{L^2} + \langle{}x,A^{\star}y\rangle{}_{L^2}\hspace{-5.5pt}
& = &\hspace{-7pt}-{\textstyle\int_a^b}\langle{}P_0x(\xi)+P_1x'(\xi),y(\xi)\rangle{}_{\mathbb{C}^d} +\langle{}x(\xi),P_0y(\xi)+P_1y'(\xi)\rangle{}_{\mathbb{C}^d}d\xi\\
& = &\hspace{-7pt}-{\textstyle\int_a^b}\langle{}P_1x'(\xi),y(\xi)\rangle{}_{\mathbb{C}^d}d\xi - {\textstyle\int_a^b}\langle{}x(\xi),P_1y'(\xi)\rangle{}_{\mathbb{C}^d}d\xi\\
& = & \langle{}P_1x(\xi),y(\xi)\rangle{}_{\mathbb{C}^d}\big|_a^b+ {\textstyle\int_a^b}\langle{}x(\xi),P_1y'(\xi)\rangle{}_{\mathbb{C}^d}d\xi - {\textstyle\int_a^b}\langle{}x(\xi),P_1y'(\xi)\rangle{}_{\mathbb{C}^d}d\xi\\
\phantom{\textstyle\int_a^b}& = & \langle{}P_1x(b),y(b)\rangle{}_{\mathbb{C}^d}-\langle{}P_1x(a),y(a)\rangle{}_{\mathbb{C}^d}\\
\phantom{\textstyle\int_a^b}& = & {\textstyle\frac{1}{2}}\bigl[\langle{}P_1(x(b)-x(a)),y(b)+y(a)\rangle{}_{\mathbb{C}^d} +\langle{}x(b)+x(a),P_1(y(b)-y(a))\rangle{}_{\mathbb{C}^d}\bigr]\\
\phantom{\textstyle\int_a^b}& = & \langle{}\Gamma_1x,\Gamma_2y\rangle{}_{\mathbb{C}^d} + \langle{}\Gamma_2x,\Gamma_1y\rangle{}_{\mathbb{C}^d}
\end{eqnarray*}
which establishes (BT-1). Let $y_1, y_2\in\mathbb{C}^d$ be given. Define $x\colon[a,b]\rightarrow\mathbb{C}^d$ via
$$
x(\xi)={\textstyle\frac{1}{\sqrt{2}}}\bigl(y_2-P_1^{-1}y_1+2\,{\textstyle\frac{\xi-a}{b-a}}P_1^{-1}y_1\bigr)
$$
for $\xi\in[a,b]$. Thus, we have $x\in H^1$, $x(a)={\textstyle\frac{1}{\sqrt{2}}}(y_2-P_1^{-1}y_1)$ and $x(b)={\textstyle\frac{1}{\sqrt{2}}}(y_2+P_1^{-1}y_1)$. Therefore
$$
\Gamma_1x = {\textstyle\frac{1}{2}} P_1 (y_2+P_1^{-1}y_1-y_2+P_1^{-1}y_1)=y_1 \;\; \text{ and } \;\;\Gamma_2x = {\textstyle\frac{1}{2}} ( y_2+P_1^{-1}y_1 +y_2-P_1^{-1}y_1)=y_2
$$
hold. This establishes (BT-2).
\end{proof}

\smallskip

For the rest of the section we fix the boundary triplet $(\mathbb{C}^d,\Gamma_1,\Gamma_2)$. Below we apply the results of Section \ref{SEC-4} to this triplet. In particular, the abstract space $H$ appearing in the definition of $\mathcal{K}$ in Section \ref{SEC-4} is from now \red{on} always the Hilbert space $\mathbb{C}^d$.  

\smallskip

\begin{lem}\label{NEW-LEM} The map $\Theta\colon\mathcal{W}\rightarrow\mathcal{K}$, $\Theta(W)=-(W_1+W_2)^{\red{-1}}(W_1-W_2)$, for $W=[W_1\;W_2]$ with $W_1$, $W_2\in\mathbb{C}^{d\times d}$, is well-defined, surjective and
$$
D((\Psi\circ\Theta)(W))=\big\{x\in H^1\:;\:W{\textstyle\columnvec{\Gamma_1x}{\Gamma_2x}}=0\big\}
$$
holds for every $W\in\mathcal{W}$.
\end{lem}
\begin{proof} We remember that the estimate $M_1\geqslant M_2$ for $M_1$, $M_2\in\mathbb{C}^{d\times{}d}$ means by definition that $M_1-M_2\geqslant0$ holds; $M\geqslant0$ for a matrix $M$ means $\langle{}M\xi,\xi\rangle{}_{\mathbb{C}^d}\geqslant0$ for all $\xi\in\mathbb{C}^d$.

\smallskip

Let $W=[W_1\;W_2]\in\mathcal{W}$ with $W_1$, $W_2\in\mathbb{C}^{d\times{}d}$ be given. We compute $W\Sigma{}W^{\star}=W_1W_2^{\star}+W_2W_1^{\star}\geqslant0$, that is $W_2W_1^{\star}\geqslant -W_1W_2^{\star}$ and $W_1W_2^{\star}\geqslant-W_2W_1^{\star}$. Therefore,
\begin{eqnarray*}
(W_1+W_2)(W_1+W_2)^{\star} &=& W_1W_2^{\star}+W_1W_2^{\star}+W_2W_1^{\star}+W_2W_2^{\star}\\
& \geqslant & W_1W_2^{\star}-W_2W_1^{\star}-W_1W_2^{\star}+W_2W_2^{\star}\\
& = & (W_1-W_2)(W_1-W_2)^{\star}\\
& \geqslant & 0.
\end{eqnarray*}
Let $\zeta\in\mathbb{C}^d$ and assume $W_1^{\star}\zeta+W_2^{\star}\zeta=0$. By the above $\langle{}(W_1-W_2)^{\star}\zeta,(W_1-W_2)^{\star}\zeta\rangle{}_{\mathbb{C}^d}=0$, i.e., $W_1^{\star}\zeta-W_2^{\star}\zeta=0$ holds. Both equations together yield $W_1^{\star}\zeta=W_2^{\star}\zeta=0$, i.e., $W\zeta=[W_1\;W_2]\zeta=0$. Since $\rk W = d$ it \red{follows that} $\zeta=0$ and we proved that $W_1+W_2$ and thus also $(W_1+W_2)^{\star}$ is invertible. Using the estimate above again, we get
$$
\Theta(W)\Theta(W)^{\star}=(W_1+W_2)^{-1}(W_1-W_2)(W_1-W_2)^{\star}((W_1+W_2)^{\star})^{-1}\leqslant\id_{\mathbb{C}^d}
$$
which shows $\|\Theta(W)\zeta\|_{\mathbb{C}^d}^2=|\langle{}\Theta(W)\Theta(W)^{\star}\zeta,\zeta\rangle{}_{\mathbb{C}^d}|\leqslant1$ for every $\zeta\in\mathbb{C}^d$, i.e., $\Theta(W)\in\mathcal{K}$. Now we compute
$$
W = {\textstyle\frac{1}{2}}\bigl[W_1+W_2+W_1-W_2\;\;W_1+W_2-W_1+W_2\bigr] = S\bigl[\Theta(W)-\id_{\mathbb{C}^d}\;\;-(\Theta(W)+\id_{\mathbb{C}^d})\bigr]
$$
with $S=-\frac{1}{2}(W_1+W_2)$. Since $S$ is invertible, we have
$$
D((\Psi\circ\Theta)(W))=\big\{x\in H^1 \:;\:[\Theta(W)-\id_{\mathbb{C}^d}\;\;-(\Theta(W)+\id_{\mathbb{C}^d})]{\textstyle\columnvec{\Gamma_1x}{\Gamma_2x}}=0\big\}=\big\{x\in H^1\:;\:W{\textstyle\columnvec{\Gamma_1x}{\Gamma_2x}}=0\big\}.
$$
Let $K\in\mathcal{K}$ be given. We put $W=[K-\id_{\mathbb{C}^d}\;\;-(K+\id_{\mathbb{C}^d})]$ and compute
$$
\Theta(W) = -(K-\id_{\mathbb{C}^d}-(K+\id_{\mathbb{C}^d}))^{-1}(K-\id_{\mathbb{C}^d}+(K+\id_{\mathbb{C}^d}))=K.
$$
Moreover we have $W\Sigma{}W^{\star}=2(\id_{\mathbb{C}^d}-KK^{\star})\geqslant0$ since $K\in\mathcal{K}$ implies $KK^{\star}\leqslant\id_{\mathbb{C}}$. Let $\zeta\in\mathbb{C}^d$ be given with
$$
W^{\star}\zeta=\bigl[(K-\id_{\mathbb{C}^d})\zeta\;\;-(K+\id_{\mathbb{C}^d})\zeta\bigr]=0.
$$
Thus, $(K-\id_{\mathbb{C}^d})\zeta=0$ and $(K+\id_{\mathbb{C}^d})\zeta=0$. Consequently, $\zeta=0$ and we obtain that $\rk W=\rk W^{\star}=d$. This establishes the surjectivity of $\Theta$.
\end{proof}

\smallskip

Now we apply the results of Section \ref{SEC-4}. We remember that we already fixed a boundary triplet and that all symbols introduced in Section \ref{SEC-4} have to be understood with respect to this particular triplet.

\smallskip

\textit{Proof\,(of Proposition \ref{PROP-PHS}).} Let $W\in\mathcal{W}$ be given. By Lemma \ref{NEW-LEM}, we have $\Theta(W)\in\mathcal{K}$ with
$$
D((\Psi\circ\Theta)(W))=\big\{x\in H^1\:;\:W{\textstyle\columnvec{\Gamma_1x}{\Gamma_2x}}=0\big\}=D(B)
$$
and $(\Psi\circ\Theta)(W)\in\mathcal{D}$\red{, which generates a $\Cnull$-semigroup by the Lumer-Phillips theorem, see Section \ref{SEC-3}}. By Proposition \ref{PROP-1} therefore $(\Psi\circ\Theta)(W)\subseteq-A^{\star}$ holds. Lemma \ref{LEM-5}, together with the equality of domains that we established already, yields
$$
(\Psi\circ\Theta)(W)x=P_0x+P_1x'=Bx
$$
for $x\in D((\Psi\circ\Theta)(W))=D(B)$.

\smallskip

For the second part let $B\colon D(B)\rightarrow L^2$ be an extension of $A$ that generates a $\Cnull$-semigroup of contractions. By the Lumer-Phillips theorem, $B$ is a maximal dissipative extension of $A$, i.e., $B\in\mathcal{D}$. By Lemma \ref{NEW-LEM} there exists $W\in\mathcal{W}$ such that $\Theta(W)=\Phi(B)\in\mathcal{K}$ holds, from whence the equality
$$
B=(\Psi\circ\Phi)(B)=(\Psi\circ\Theta)(W)
$$
follows. The first part of the proof shows that $B$ is of the desired form.\hfill\qed


\section{Notes and References}\label{SEC-7}

Below we give the references to those papers and books in which the results presented in the preceding sections have been originally proved. In addition, we comment these results and the methods we used for their proofs. We give detailed references and point out alternative approaches.

\bigskip

\begin{center}
Section \ref{SEC-2}.
\end{center}

Our notion of a dissipative operator in Definition \ref{dissipative}(i) is widely used in the  literature on $\Cnull$-semigroups, cf.~Engel, Nagel \cite{EN} and Pazy \cite{Pazy}. In the symmetric situation, and in particular in the book \cite{GG} by Gorbachuk, Gorbachuk, which is often used as a reference for the theory of boundary triplets, $A$ is by definition dissipative if $\Im\langle{}Ax,x\rangle{}\geqslant0$ holds for all $x\in D(A)$. This corresponds to passing from the operator $A$ to $iA$. We like to refer also to the book \cite{Nagy} by Sz.-Nagy, Foias, Bercovici, K{\'e}rchy, in particular to \cite[IV.4]{Nagy}.

\smallskip

Our terminology of maximal dissipativity in Definition \ref{dissipative}(ii) is defined via extensions. In \cite[Definition 3.3.1]{Pazy} the notion of m-dissipative operators is used. An operator $A$ is \textit{m-dissipative} if $A$ is dissipative and $\lambda-A$ is surjective for some or, equivalently, all $\lambda>0$. The relation between these two, a priori different, notions depends on the underlying situation. In the previous sections the operator $A$ under consideration was always assumed to be densely defined. Dissipative, maximal dissipative and m-dissipative operators can however be studied without this assumption. Indeed, even if $X$ is only a Banach space, one can define the three notions: An operator $A$ is then by definition 
dissipative if the condition in Lemma \ref{LEM}(iv) holds, cf.~\cite[Definition II.3.13]{EN}. If now $A$ is densely defined, then $A$ is maximal dissipative if and only if $A$ is m-dissipative. For Hilbert spaces this follows from the material presented in Section \ref{SEC-3} and our comments below. For Banach spaces we refer to \cite[Definition II.3.13 and Proposition II.3.23]{EN} together with \cite[Section 1.4 and p.\,81]{Pazy}. If, on the other hand, $X$ is a Hilbert space then every m-dissipative operator is automatically densely defined, in fact this is even true for reflexive Banach spaces \cite[Corollary II.3.20]{EN}. In this situation we thus derive that every m-dissipative operator is maximal dissipative. The converse is finally not true even if $X$ is Hilbert: The results \cite[footnote on p.\,201]{P1959} provide an example of a maximal dissipative operator which is not densely defined by \cite[Lemma 1.1.3]{P1959} and consequently cannot be m-dissipative.

\smallskip

The decomposition in Lemma \ref{DECOMP-LEM} can alternatively be proved by adapting the proof of \cite[Lemma X.2.13]{C} where symmetric operators are considered. Our proof is an adaption of \cite[Proposition 3.7]{S}.

\bigskip

\begin{center}
Section \ref{SEC-3}.
\end{center}

The version of the Lumer-Phillips theorem in Section \ref{SEC-3} is adjusted to our definition of maximal dissipativity. A popular alternative is the following \cite[Theorem 4.3]{Pazy}: If $A$ is densely defined then the following are equivalent.
\begin{compactitem}\vspace{3pt}
\item[(i)] $A$ is dissipative and $\lambda-A$ is surjective for some or, equivalently, for all $\lambda>0$,\vspace{2pt}
\item[(ii)] $A$ generates a $\Cnull$-semigroup of contractions. 
\end{compactitem}\vspace{3pt}
In \cite[Section II.3.b]{EN} more sophisticated versions can be found. The reader familiar with the version stated above will find it easy to extract the equivalence of
\begin{compactitem}\vspace{3pt}
\item[(iii)] $A$ is maximal dissipative in the sense of Definition \ref{dissipative}(i),
\end{compactitem}\vspace{3pt}
and (ii) from the proof given in Section \ref{SEC-3}.

\bigskip

\begin{center}
Section \ref{SEC-4} and Section \ref{SEC-5}.
\end{center}

For precise historical information on the symmetric counterparts of the statements explained in Section \ref{SEC-2}--\ref{SEC-5} we refer to Gorbachuk, Gorbachuk \cite[p.~320]{GG}. 

\smallskip

Let $X$ be a Hilbert space and $A\colon D(A)\rightarrow X$ be a closed, densely defined and skew-symmetric operator. Define $L\colon \ran(1-A)\rightarrow X$ via $L(x-Ax)=x+Ax$. Then by Phillips \cite[p.\,199-201]{P1959} there is a one-to-one correspondence
\begin{equation}\label{PH}
\bigg\{\hspace{-3pt}\begin{array}{c}
B\colon D(B)\rightarrow X\\
A\subseteq B \; \& \; B \text{ dissipative}
\end{array}\hspace{-3pt}\bigg\}
\begin{array}{c}\vspace{-3pt}
\stackrel{\Phi}{\longrightarrow}\\\vspace{-5pt}
\stackrel{\textstyle\longleftarrow}{\scriptstyle\Psi}\vspace{6pt}
\end{array}
\bigg\{\hspace{-3pt}\begin{array}{c}
K\colon D(K)\rightarrow X\\
L\subseteq K \; \& \; K \text{ contractive}
\end{array}\hspace{-3pt}\bigg\}
\end{equation}
where the maps $\Phi$ and $\Psi$ are given by $\Phi(B)\colon\!\ran(1-B)\rightarrow X,\; \Phi(B)(x-Bx)=x+Bx$ and $\Psi(K)\colon\!\ran(1+K)\rightarrow X, \;\Psi(K)(Kx+x)=Kx-x$. The maximal dissipative extensions correspond exactly to those contractions which are defined on the whole space. The method of boundary triplets can be regarded as a refinement of this classical result: In \eqref{PH} the contractions which are used to parametrize the dissipative extensions are defined on the initial Hilbert space $X$, which is typically an infinite dimensional Hilbert space. But the above suggests that for the parametrization a smaller space may be enough. Since $\Phi(B)|_{\ran(1-A)}=L$ holds for every extension $B$ of $A$ it is reasonable to give the space $X/\ran(1-A)=\ker(1-A^{\star})=:H$ a try as domain of the contractions. Killing \red{that} part of the domains on which there is no choice for the values of $B$ means to consider $D(-A^{\star})/D(A)=\ker(1-A^{\star})\oplus\ker(1+A^{\star})$, cf.~Lemma \ref{DECOMP-LEM}. 
This discrepancy gives rise to the characterization that the re-parametrization works if $\ker(1-A^{\star})\cong\ker(1+A^{\star})$ holds. The precise details are contained in Section 
\ref{SEC-5}. Note that in the book \cite{GG} by Gorbachuk, Gorbachuk only the sufficiency part of Theorem \ref{EX-THM} is given; the necessity it taken from Schm\"udgen \cite[Proposition 14.5]{S}.

\bigskip

\begin{center}
Section \ref{SEC-6}.
\end{center}

We refer to the book \cite{JZ} by Jacob, Zwart for a detailed introduction to port-Hamiltonian systems \red{which also includes} an introduction to semigroup theory and many examples and applications from control and systems theory. The exposition in \cite{JZ} avoids the notion of boundary triplets and uses a direct Lumer-Phillips argument. The technical ingredients for this are however the same as in the approach given in Section \ref{SEC-6}. The latter follows mainly Le Gorrec, Zwart, Maschke \cite{GZM2005}, only some details are different: Firstly, in the latter paper the differential operator is of order $N\geqslant1$ whereas we restricted ourselves to the case $N=1$ to make the computations somewhat easier. Secondly, in \cite{GZM2005} the matrices $P_i$ are real and in our case they are complex. For further literature, in particular on applications of port-Hamiltonian systems in control theory, we refer to the references given in Section \ref{SEC-1}. In particular, we refer to \cite{GZM2005, JZ} for physical interpretations which show that the special form of boundary conditions that are used in Theorem \ref{PHS-THM} is reasonable for applications.

\bigskip

\footnotesize

{\sc Acknowledgements. }The author likes to thank B.~Augner, A.~B\'{a}tkai, B.~Farkas and B.~Jacob for several helpful comments and discussions. In addition, he likes to thank B.~Farkas for carefully reading this manuscript. \red{Finally the author likes to thank the anonymous referees for their corrections and very valuable comments.}  

\normalsize

\bigskip

\normalsize

\bibliographystyle{amsplain}

\begin{thebibliography}{10}

\bibitem{A2012}
Y.~Arlinski{\u\i}, \emph{Boundary triplets and maximal accretive extensions of sectorial operators}, Operator methods for boundary value problems, London Math. Soc. Lecture Note Ser., vol. 404, Cambridge Univ. Press, Cambridge, 2012, pp.~35--72.

\bibitem{AKS2012}
D.~Z. Arov, M.~Kurula, and O.~J. Staffans, \emph{Boundary control state/signal systems and boundary triplets}, Operator methods for boundary value problems, London Math. Soc. Lecture Note Ser., vol. 404, Cambridge Univ. Press, Cambridge, 2012, pp.~73--85.

\bibitem{AJ2014}
B.~Augner and B.~Jacob, \emph{Stability and stabilization of infinite-dimensional linear port-{H}amiltonian systems}, Evol. Equ. Control Theory \textbf{3} (2014), no.~2, 207--229.

\bibitem{BHdSW2011}
J.~Behrndt, S.~Hassi, H.~de~Snoo, and R.~Wietsma, \emph{Square-integrable solutions and {W}eyl functions for singular canonical systems}, Math. Nachr. \textbf{284} (2011), no.~11-12, 1334--1384.

\bibitem{BL2007}
J.~Behrndt and M.~Langer, \emph{Boundary value problems for elliptic partial differential operators on bounded domains}, J. Funct. Anal. \textbf{243} (2007), no.~2, 536--565.

\bibitem{BL2012}
J.~Behrndt and M.~Langer, \emph{Elliptic operators, {D}irichlet-to-{N}eumann maps and quasi boundary triples}, Operator methods for boundary value problems, London Math. Soc. Lecture Note Ser., vol. 404, Cambridge Univ. Press, Cambridge, 2012, pp.~121--160.

\bibitem{C}
J.~B. Conway, \emph{A course in functional analysis}, second ed., vol.~96, Springer-Verlag, New York, 1990.

\bibitem{DHMdS2006}
V.~Derkach, S.~Hassi, M.~Malamud, and H.~de~Snoo, \emph{Boundary relations and their {W}eyl families}, Trans. Amer. Math. Soc. \textbf{358} (2006), no.~12, 5351--5400.

\bibitem{DHMdS2009}
V.~Derkach, S.~Hassi, M.~Malamud, and H.~de~Snoo, \emph{Boundary relations and generalized resolvents of symmetric operators}, Russ. J. Math. Phys. \textbf{16} (2009), no.~1, 17--60.

\bibitem{DHMdS2012}
V.~Derkach, S.~Hassi, M.~Malamud, and H.~de~Snoo, \emph{Boundary triplets and {W}eyl functions. {R}ecent developments}, Operator methods for boundary value problems, London Math. Soc. Lecture Note Ser., vol. 404, Cambridge Univ. Press, Cambridge, 2012, pp.~161--220.

\bibitem{DM1991}
V.~A. Derkach and M.~M. Malamud, \emph{Generalized resolvents and the boundary value problems for {H}ermitian operators with gaps}, J. Funct. Anal. \textbf{95} (1991), no.~1, 1--95.

\bibitem{DM1995}
V.~A. Derkach and M.~M. Malamud, \emph{The extension theory of {H}ermitian operators and the moment problem}, J. Math. Sci. \textbf{73} (1995), no.~2, 141--242, Analysis. 3.

\bibitem{EN}
K.-J. Engel and R.~Nagel, \emph{One-parameter semigroups for linear evolution equations}, Springer-Verlag, New York, 2000.

\bibitem{GG}
V.~I. Gorbachuk and M.~L. Gorbachuk, \emph{Boundary value problems for operator differential equations}, Mathematics and its Applications (Soviet Series), vol.~48, Kluwer Academic Publishers Group, Dordrecht, 1991.

\bibitem{GZM2005}
Y.~Le Gorrec, H.~J. Zwart, and B.~Maschke, \emph{Dirac structures and boundary control systems associated with skew-symmetric differential operators}, SIAM J. Control Optim. \textbf{44} (2005), no.~5, 1864--1892.

\bibitem{JMZ2015}
B.~Jacob, K.~Morris, and H.~Zwart, \emph{$\text{C}_0$-semigroups for hyperbolic partial differential equations on a one-dimensional spatial domain}, Journal of Evolution Equations (2015), to appear, DOI: 10.1007/s00028-014-0271-1.

\bibitem{JZ}
B.~Jacob and H.~J. Zwart, \emph{Linear port-{H}amiltonian systems on infinite-dimensional spaces}, Operator Theory: Advances and Applications, vol. 223, Birkh\"auser/Springer Basel AG, Basel, 2012, Linear Operators and Linear Systems.

\bibitem{KZ2015}
M.~Kurula and H.~Zwart, \emph{Linear wave systems on {$n$}-{D} spatial domains}, Internat. J. Control \textbf{88} (2015), no.~5, 1063--1077.

\bibitem{M1992}
M.~M. Malamud, \emph{On a formula for the generalized resolvents of a non-densely defined {H}ermitian operator}, Ukra\"\i n. Mat. Zh. \textbf{44} (1992), no.~12, 1658--1688.

\bibitem{MS2007}
J.~Malinen and O.~J. Staffans, \emph{Impedance passive and conservative boundary control systems}, Complex Anal. Oper. Theory \textbf{1} (2007), no.~2, 279--300.

\bibitem{Pazy}
A.~Pazy, \emph{Semigroups of linear operators and applications to partial differential equations}, vol.~44, Springer-Verlag, New York, 1983.

\bibitem{P1959}
R.~S. Phillips, \emph{Dissipative operators and hyperbolic systems of partial differential equations}, Trans. Amer. Math. Soc. \textbf{90} (1959), 193--254.

\bibitem{R}
W.~Rudin, \emph{Functional analysis}, second ed., International Series in Pure and Applied Mathematics, McGraw-Hill, Inc., New York, 1991.

\bibitem{S}
K.~Schm{\"u}dgen, \emph{Unbounded self-adjoint operators on {H}ilbert space}, Graduate Texts in Mathematics, vol. 265, Springer, Dordrecht, 2012.

\bibitem{Nagy}
B.~Sz.-Nagy, C.~Foias, H.~Bercovici, and L.~K{\'e}rchy, \emph{Harmonic analysis of operators on {H}ilbert space}, second ed., Universitext, Springer, New York, 2010.

\bibitem{VZGM2009}
J.~A. Villegas, H.~Zwart, Y.~Le Gorrec, and B.~Maschke, \emph{Exponential stability of a class of boundary control systems}, IEEE Trans. Automat. Control \textbf{54} (2009), no.~1, 142--147.

\end{thebibliography}

\end{document}